\documentclass[11pt,letterpaper]{amsart}
\usepackage[dvipsnames]{xcolor}
\usepackage[T1]{fontenc}
\usepackage{amsmath, amssymb,amsthm}
\usepackage[extra]{tipa}
\usepackage{lettrine}
\usepackage[Sonny]{fncychap}
\usepackage[english]{babel}
\usepackage{amsmath, amssymb,amsthm}
\usepackage[all]{xy}
\usepackage{tabularx}
\usepackage[applemac]{inputenc}
\usepackage{graphicx} 
\usepackage[dvipsnames]{xcolor}
\usepackage{hyperref}
\usepackage[applemac]{inputenc}
\usepackage{enumitem}
\usepackage{ bbold }
\usepackage{prerex}
\usetikzlibrary{fit}
\usepackage{pdflscape}
\usepackage{mathtools}
\usepackage{tikz-cd}
\DeclarePairedDelimiter\abs{\lvert}{\rvert}

\newcommand\smvee{\raise0.9ex\hbox{$\scriptscriptstyle\vee$}}
\makeatletter
\def\oversortoftilde#1{\mathop{\vbox{\m@th\ialign{##\crcr\noalign{\kern3\p@}%
      \sortoftildefill\crcr\noalign{\kern3\p@\nointerlineskip}%
      $\hfil\displaystyle{#1}\hfil$\crcr}}}\limits}

\def\sortoftildefill{$\m@th \setbox\z@\hbox{$\braceld$}%
  \braceld\leaders\vrule \@height\ht\z@ \@depth\z@\hfill\braceru$}

\makeatother

\usepackage{mathrsfs}

    \DeclareFontFamily{U}{wncy}{}
    \DeclareFontShape{U}{wncy}{m}{n}{<->wncyr10}{}
    \DeclareSymbolFont{mcy}{U}{wncy}{m}{n}
    \DeclareMathSymbol{\Sh}{\mathord}{mcy}{"58}

\DeclareMathOperator{\Norm}{Norm}
\DeclareMathOperator{\Ker}{Ker}
\DeclareMathOperator{\Pic}{Pic}
\DeclareMathOperator{\Ind}{Ind}
\DeclareMathOperator{\Q}{\mathbf{Q}}
\DeclareMathOperator{\Z}{\mathbf{Z}}
\DeclareMathOperator{\C}{\mathbf{C}}
\DeclareMathOperator{\Gal}{Gal}

\DeclareMathOperator{\Spec}{Spec}
\DeclareMathOperator{\Sel}{Sel}
\DeclareMathOperator{\Hom}{Hom}

\DeclareMathOperator{\Tr}{Tr}
\DeclareMathOperator{\fppf}{fppf}
\DeclareMathOperator{\etale}{et}

\DeclareMathOperator{\Frob}{Frob}
\DeclareMathOperator{\SL}{SL}
\DeclareMathOperator{\GL}{GL}

\DeclareMathOperator{\rk}{rk}

\DeclareMathOperator{\Eis}{Eis}
\newcommand{\cf}{\textit{cf. }}
\newcommand{\ie}{\textit{i.e. }}

\newcommand{\eg}{\textit{eg. }}
\theoremstyle{definition}

\newtheorem*{fact}{Fact}

\newtheorem{rems}{Remarks}[section]

\theoremstyle{plain}
\newtheorem{thm}{Theorem}[section]

\newtheorem{lem}[thm]{Lemma}
\newtheorem{corr}[thm]{Corollary}

\newtheorem{prop}[thm]{Proposition}

\newtheorem{innercustomthm}{Hypothesis}
\newenvironment{hypo}[1]
  {\renewcommand\theinnercustomthm{#1}\innercustomthm}
  {\endinnercustomthm}

\title{On the arithmetic of special values of $L$-functions for certain abelian varieties with a rational isogeny}
\author{Emmanuel Lecouturier and Jun Wang}

\begin{document}
\maketitle

\begin{abstract} 
Let $N$ and $p$ be primes $\geq 5$ such that $p \mid \mid N-1$. In this situation, Mazur \cite{Mazur_Eisenstein} defined and studied the $p$-\emph{Eisenstein quotient} $\tilde{J}^{(p)}$ of $J_0(N)$. We prove a kind of modulo $p$ version of the Birch and Swinnerton-Dyer conjecture for the ``$p$-Eisenstein part'' of even quadratic twists of $\tilde{J}^{(p)}$. Our result is the analogue for even quadratic twists of a result of Mazur \cite{Mazur_1979} concerning odd quadratic twists.
\end{abstract}
\tableofcontents

\section{Introduction}\label{Section_introduction}

Let $N$ and $p$ be prime numbers $\geq 5$ such that $p \mid N-1$. In this situation, Mazur proved in his seminal work \cite{Mazur_Eisenstein} that there is a cuspidal eigenform $f \in S_2(\Gamma_0(N))$ which is Eisenstein modulo $p$. More precisely, there is a prime above $\mathfrak{p}$ in the ring of integer $\mathcal{O}_f$ of the Hecke field of $f$ such that 
\begin{equation}\label{eq_Eis_congruence}
f \equiv E_2 \text{ (modulo }\mathfrak{p}\text{),}
\end{equation}
where $$E_2 = \frac{N-1}{24} + \sum_{n\geq 1} \left(\sum_{d \mid n, \gcd(d,N)=1} d\right)\cdot q^n$$
is the unique Eisenstein series of weight $2$ and level $\Gamma_0(N)$. Actually, there may be several such cuspidal eigenforms $f$. We denote by $S_2^{\Eis}(N,p)$ the (finite and non-empty) set of cuspidal eigenforms in $S_2(\Gamma_0(N))$ which satisfy a congruence of the form (\ref{eq_Eis_congruence}).

In order to study in more details (\ref{eq_Eis_congruence}), Mazur studied the \emph{Eisenstein ideal}. Let $\mathbb{T}^0$ be the cuspidal Hecke algebra over $\Z$ acting on $S_2(\Gamma_0(N))$. Let $I \subset \mathbb{T}^0$ be Mazur's Eisenstein ideal, generated by the Hecke operators $T_{\ell}-\ell-1$ for primes $\ell \neq N$ as well as $U_N-1$. Mazur proved that $\mathbb{T}^0/I$ is cyclic of order the numerator of $\frac{N-1}{12}$ (\cf \cite[Proposition II.9.7]{Mazur_Eisenstein}). Let $\mathfrak{P}$ be the maximal ideal of $\mathbb{T}^0$ generated by $I$ and $p$. We denote by $\mathbb{T}_{\mathfrak{P}}^0$ the $\mathfrak{P}$-adic completion of $\mathbb{T}^0$.  More generally, if $M$ is a $\mathbb{T}^0$-module then we denote by $M_{\mathfrak{P}}$ the $\mathfrak{P}$-adic completion of $M$, \ie $M_{\mathfrak{P}} = M\otimes_{\mathbb{T}^0} \mathbb{T}_{\mathfrak{P}}^0$. 

For simplicity, in all this paper we will assume that $p^2 \nmid N-1$. In this case, Mazur proved in \cite[Proposition II.19.1]{Mazur_Eisenstein} that $\mathbb{T}_{\mathfrak{P}}^0$ is a DVR. Many of our results could be formulated and proved without this assumption, but it would complicate the proofs and the notation.

Mazur defined in \cite[Definition II.10.4]{Mazur_Eisenstein} the $p$-\emph{Eisenstein quotient} $\tilde{J}^{(p)}$ of $J_0(N)$. It is the quotient of the Jacobian $J_0(N)$ of the modular curve $X_0(N)$ corresponding, roughly speaking, to the eigenforms in $S_2^{\Eis}(N,p)$. More precisely, Mazur defines it as follows:
$$\tilde{J}^{(p)} := J_0(N)/\gamma_{\mathfrak{P}}\cdot J_0(N)$$
where 
\begin{equation}\label{eq_gamma_P}
\gamma_{\mathfrak{P}} = \Ker(\mathbb{T}^0 \rightarrow \mathbb{T}_{\mathfrak{P}}^0) \text{ .}
\end{equation}
Note that 
$$\dim \tilde{J}^{(p)} = \# S_2^{\Eis}(N,p) \text{ .}$$

Mazur studied in \cite{Mazur_1979} a modulo $\mathfrak{P}$ version of the Birch and Swinnerton-Dyer conjecture for the base change of $\tilde{J}^{(p)}$ to imaginary quadratic fields, or equivalently for quadratic twists of $\tilde{J}^{(p)}$ over $\Q$ by odd quadratic Dirichlet characters. Let us recall Mazur's result. In this paper, fix a quadratic field $K$ with discriminant $D$ and associated quadratic character $\chi_D : (\Z/D\Z)^{\times} \rightarrow \{\pm 1\}$. Let $h(K)$ be the class number of the ring of integers $\mathcal{O}_K$ of $K$. We denote by $\tilde{J}^{(p)}_{D}$ the quadratic twist of $\tilde{J}^{(p)}$ corresponding to $K$.  Note that both $\tilde{J}^{(p)}$ and $\tilde{J}^{(p)}_{D}$ have a canonical action of $\mathbb{T}^0/\gamma_{\mathfrak{P}}$.

We would like to talk about the ``$\mathfrak{P}$-component of the algebraic part of the $L$-function of $\tilde{J}^{(p)}_{D}$ at $s=1$''. Mazur considers in \cite[Chap. IV, \S 1]{Mazur_1979} the universal special value $\Lambda(U_{\mathfrak{P}}, \chi_D)$ (when $D<0$, \ie when $\chi_D$ is odd). See section \ref{section_L_modSymb} (\ref{eq_def_univ_L}) for a precise definition. Let us just say here that $\Lambda(U_{\mathfrak{P}}, \chi_D) \in H_1(X_0(N), \Z)_{\mathfrak{P}}^{\pm}$ where $\chi_D(-1) = \pm 1$.

\begin{thm}\cite[Chap. IV, \S 1, Theorem]{Mazur_1979}\label{Mazur_main_thm}
Assume $D<0$, $N$ is inert in $K$ and $p\nmid D$. We have $\Lambda(U_{\mathfrak{P}}, \chi_D) \in I \cdot H_1(X_0(N), \Z)_{\mathfrak{P}}^{-}$ if and only if the Mordell--Weil rank of $\tilde{J}^{(p)}_{D}$ is $>0$ or the completion at $\mathfrak{P}$ of the Tate--Shafarevich group of $\tilde{J}^{(p)}_{D}$ is non-trivial.
\end{thm}

We would like to have a statement more closely related to the BSD conjecture for $\tilde{J}^{(p)}_{D}$. One can deduce such a statement only under the following assumption:
\begin{hypo}{H}\label{hypo}
$$\#S_2^{\Eis}(N,p) = \rk_{\Z_p} \mathbb{T}^0_{\mathfrak{P}} \text{ .}$$
\end{hypo}
The right hand side being an important invariant, we follow Mazur let $$g_p := \rk_{\Z_p} \mathbb{T}^0_{\mathfrak{P}} \geq 1 \text{ .}$$
The inequality $\#S_2^{\Eis}(N,p) \geq g_p$ always holds, but the equality occurs rather rarely. Using tables of Naskr\k{e}cki (which are tables extending the ones from \cite{Naskrecki} and kindly shared privately by Naskr\k{e}cki with the first author), one finds that the set of triples $(N,p,g_p)$ with $p\geq 5$, $p \mid\mid N-1$ and $N \leq 14251$ such that Hypothesis \ref{hypo} holds is
$$\{ (11,5,1), (31, 5, 2), (211, 5, 2), (1871, 5, 2), (4621, 5, 2), (9931, 5, 2)\} \text{ .}$$
Thus, in some sense, the simplest example for which our hypothesis holds is $N=11$ and $p=5$, in which case $\tilde{J}^{(p)}$ is the elliptic curve $X_0(11)$. We refer to section \ref{section_L_modSymb} for a reformulation of Hypothesis \ref{hypo}.

 Let  $L(\tilde{J}^{(p)}_{D},s)$ and $\Omega_D>0$ be the $L$-function of $\tilde{J}^{(p)}_{D}$ and its real N\'eron period respectively. One can deduce from Theorem \ref{Mazur_main_thm} the following result toward BSD.

\begin{corr}\label{Mazur_main_cor}
Assume Hypothesis \ref{hypo}. Assume $D<0$, $N$ is inert in $K$ and $p\nmid D$. Then $p$ divides the numerator of $\frac{L(\tilde{J}^{(p)}_{D},1)}{\Omega_D}$ if and only if the Mordell--Weil rank of $\tilde{J}^{(p)}_{D}$ is $>0$ or the $p$-part of the Tate--Shafarevich group of $\tilde{J}^{(p)}_{D}$ is non-trivial.
\end{corr}

\begin{rems}\label{rem_Mazur_main_thm}
\begin{enumerate}
\item\label{rem_Mazur_main_thm_i} 
One can easily check that the above result is as predicted by the BSD conjecture (using the fact that $p$ does not divide the Tamagawa numbers of $\tilde{J}^{(p)}_{D}$ and the latter does not have any rational $p$-torsion). We may consider this result as a kind of modulo $p$ version of the BSD conjecture for $\tilde{J}^{(p)}_{D}$. 
\item\label{rem_Mazur_main_thm_ii} The assumption that $N$ is inert in $K$ is equivalent to the fact that the sign of $L(\tilde{J}^{(p)}_{D},s)$ is $1$.
\end{enumerate}
\end{rems}

The following is our main result.

\begin{thm}\label{main_thm}
Assume $D>0$, $N$ splits in $K$ and $p \nmid D$. Then we have $\Lambda(U_{\mathfrak{P}}, \chi_D) \in I \cdot H_1(X_0(N), \Z)_{\mathfrak{P}}^{+}$. Furthermore, we have $\Lambda(U_{\mathfrak{P}}, \chi_D) \in I^2 \cdot H_1(X_0(N), \Z)_{\mathfrak{P}}^{\pm}$ if and only if the Mordell--Weil rank of $\tilde{J}^{(p)}_{D}$ is $>0$ or the completion at $\mathfrak{P}$ of the Tate--Shafarevich group of $\tilde{J}^{(p)}_{D}$ is non-trivial.
\end{thm}

As before, one can deduce the following result toward BSD. 

\begin{corr}\label{main_cor}
Assume Hypothesis \ref{hypo}. Assume $D>0$, $N$ splits in $K$ and $p \nmid D$. Then $p$ divides the numerator of $\frac{L(\tilde{J}^{(p)}_{D},1)}{\Omega_D}$. Furthermore, $p^2$ divides the numerator of $\frac{L(\tilde{J}^{(p)}_{D},1)}{\Omega_D}$ if and only if the Mordell--Weil rank of $\tilde{J}^{(p)}_{D}$ is $>0$ or the $p$-part of the Tate--Shafarevich group of $\tilde{J}^{(p)}_{D}$ is non-trivial.
\end{corr}

The assumption that $N$ splits in $K$ is equivalent to the fact that the sign of $L(\tilde{J}^{(p)}_{D},s)$ is $1$. As above, this result is predicted by the BSD conjecture, and may be considered as a modulo $p$ version of the BSD conjecture for $\tilde{J}^{(p)}_{D}$. Our strategy is similar to Mazur's, which we now recall. For simplicity, let us assume $\#S_2^{\Eis}(N,p) = g_p=1$ (in particular, Hypothesis \ref{hypo} is satisfied), the general case being similar. Mazur proceeds in two main steps. 

\begin{itemize}
\item Step 1: congruence formula for the $L$-value. Mazur gives a congruence formula modulo $I$ (the Eisenstein ideal) for $\Lambda(U_{\mathfrak{P}}, \chi_D)$. He uses modular symbols and a multiplicity one result to obtain such a congruence. Let us explain intuitively why one can expect such a congruence formula. Let $f \in S_2(\Gamma_0(N))$ such that $f\equiv E_2 \text{ (modulo }p\text{)}$ as above. We can expect a congruence between the ``algebraic parts'' of $L(f, \chi_D, 1)$ and $L(E_2, \chi_D,1)$ (such congruences have been established in some cases \eg by \cite{Vatsal}).  We have $$L(E_2, \chi_D, 1) = L(\chi_D, 1)\cdot L(\chi_D, 0) \text{ ,}$$ which is proportional to $L(\chi_D, 0)^2 = B_{1, \chi_D}^2$ by the functional equation. By the class number formula, $B_{1, \chi_D}$ is essentially the class number $h(K)$ of $K$ (since $K$ is imaginary in Mazur's result). Thus, one would expect that the algebraic part of $L(f, \chi_D, 1)$ is proportional to $h(K)^2$ modulo $p$. This is what Mazur actually shows using modular symbols in \cite[II \S 7 Proposition]{Mazur_1979}.

\item Step 2: Modulo $p$ Selmer group computation. Mazur performs an Eisenstein descent using flat cohomology. He is able to show in \cite[\S 1 Step 2]{Mazur_1979} that $p$ does not divide $h(K)$ if and only if both the rank of the Mordell--Weil group of $\tilde{J}^{(p)}_D$ and the $p$-part of its Tate--Shafarevich group are zero. Let us again give a heuristic explanation to this result. Since $\tilde{J}^{(p)}_D(\Q)$ does not have any $p$-torsion, this triviality of the rank and the $p$-part of $\Sh$ is equivalent to $\Sel_p(\tilde{J}^{(p)}_D/\Q)=0$, where $\Sel_p$ is the $p$-Selmer group, a subgroup of the Galois cohomology group $H^1(\Gal(\overline{\Q}/\Q), \tilde{J}^{(p)}_D[p])$ (and by restriction a subgroup of $H^1(\Gal(\overline{\Q}/K), \tilde{J}^{(p)}_D[p])$). We then have an isomorphism of $\Gal(\overline{\Q}/K)$-modules $$\tilde{J}^{(p)}_D[p] \simeq  \tilde{J}^{(p)}[p] \simeq \mu_p \oplus \Z/p\Z \text{.}$$ Thus, $\Sel_p(\tilde{J}^{(p)}_D/\Q)$ is a subgroup of $$H^1(\Gal(\overline{\Q}/K), \mu_p) \oplus H^1(\Gal(\overline{\Q}/K), \Z/p\Z) \text{ .}$$ The Selmer conditions actually makes $\Sel_p(\tilde{J}^{(p)}_D/\Q)$ a subgroup of $$\Hom(\Pic(A), \Z/p\Z) \oplus H^1(\Spec(A), \mu_p)$$ where $A = \mathcal{O}_K[1/N]$.  The group $H^1(\Spec(A), \mu_p)$ is an extension of $\Pic(A)[p]$ by $A^{\times} \otimes \Z/p\Z$. Since $A^{\times} \otimes \Z/p\Z$ is generated by $N$ which fixed by $\Gal(K/\Q)$, one may expect that $\Sel_p(\tilde{J}^{(p)}_D/\Q)$ is non-zero if and only if $p$ does not divide $h(K)$.
\end{itemize}
In conclusion, Mazur proves the following result, which implies Theorem \ref{Mazur_main_thm}:

\begin{thm}\cite[Chap. IV, \S 1, Theorem]{Mazur_1979}\label{Mazur_main_thm_bis}
Assume $D<0$, $N$ is inert in $K$ and $p \nmid D$. Then the following assertions are equivalent:
\begin{enumerate}
\item $\Lambda(U_{\mathfrak{P}}, \chi_D) \in I \cdot H_1(X_0(N), \Z)_{\mathfrak{P}}^{-}$.
\item $p$ divides $h(K)$.
\item  $\Sel_p(\tilde{J}^{(p)}_D/\Q)_{\mathfrak{P}} \neq 0$.
\end{enumerate}
\end{thm}

One can deduce the following result under Hypothesis \ref{hypo}.
\begin{corr}\label{Mazur_main_cor_bis}
Assume Hypothesis \ref{hypo}.  Assume $D<0$, $N$ is inert in $K$ and $p \nmid D$. Then the following assertions are equivalent:
\begin{enumerate}
\item $p$ divides the numerator of $\frac{L(\tilde{J}^{(p)}_{D},1)}{\Omega_D}$.
\item $p$ divides $h(K)$.
\item  $\Sel_p(\tilde{J}^{(p)}_D/\Q) \neq 0$.
\end{enumerate}
\end{corr}

Similarly, we prove the following strengthening of Theorem \ref{main_thm}:

\begin{thm}\label{main_thm_bis}
Assume $D>0$, $N$ splits in $K$ and $p \nmid D$. Let $u(K)$ be a fundamental unit in $\mathcal{O}_K^{\times}$. Then we have $\Lambda(U_{\mathfrak{P}}, \chi_D) \in I \cdot H_1(X_0(N), \Z)_{\mathfrak{P}}^{+}$, and the following assertions are equivalent:
\begin{enumerate}
\item\label{main_thm_bis_1} $\Lambda(U_{\mathfrak{P}}, \chi_D) \in I^2 \cdot H_1(X_0(N), \Z)_{\mathfrak{P}}^{+}$.
\item\label{main_thm_bis_2} $u(K)^{h(K)}$ is a $p$th power modulo any prime dividing $N$ in $K$.
\item\label{main_thm_bis_3}  $\Sel_p(\tilde{J}^{(p)}_D/\Q)_{\mathfrak{P}} \neq 0$.
\end{enumerate}
\end{thm}

One can deduce the following result under Hypothesis \ref{hypo}.
\begin{corr}\label{main_cor_bis}
Assume Hypothesis \ref{hypo}. Assume $D>0$, $N$ splits in $K$ and $p \nmid D$. Let $u(K)$ be a fundamental unit in $\mathcal{O}_K^{\times}$. Then $p$ divides the numerator of $\frac{L(\tilde{J}^{(p)}_{D},1)}{\Omega_D}$, and the following assertions are equivalent:
\begin{enumerate}
\item\label{main_cor_bis_1} $p^2$ divides the numerator of $\frac{L(\tilde{J}^{(p)}_{D},1)}{\Omega_D}$.
\item\label{main_cor_bis_2} $u(K)^{h(K)}$ is a $p$th power modulo any prime dividing $N$ in $K$.
\item\label{main_cor_bis_3}  $\Sel_p(\tilde{J}^{(p)}_D/\Q)\neq 0$.
\end{enumerate}
\end{corr}

Let us explain the difficulties which arise when trying to follow Mazur's approach.
\begin{itemize}
\item Step 1 (the congruence formula for the $L$-value). Since $\chi_D$ is even in our case, we have $L(\chi_D, 0)=0$. Thus, the na\"ive congruence argument above yields that $p$ always divides the algebraic part of $L(f, \chi_D, 1)$. This can easily be proved rigorously using a congruence for even modular symbols due to Mazur (\cf \cite[p. 211]{Mazur_1979}). The fact that $p$ always divides $\frac{L(\tilde{J}^{(p)}_{D},1)}{\Omega_D}$ if $D>0$ is actually predicted by the BSD conjecture: the $N$th Tamagawa number of $\tilde{J}^{(p)}_{D}$ is the same as the $N$th Tamagawa number of $\tilde{J}^{(p)}$ since $N$ splits in $K$, and the latter is divisible by $p$ by \cite[Theorem A.1]{Mazur_Eisenstein}. In order to understand $\frac{L(\tilde{J}^{(p)}_{D},1)}{\Omega_D}$ modulo $p^2$, one way would be to have a congruence formula for even modular symbols modulo $p^2$. Such a congruence has been given in \cite[Theorem 1.12]{LW} using Sharifi's conjecture at level $\Gamma_1(N)$. However, this formula is given in terms of Manin symbols and is not enough for our purposes. One needs to use Sharifi's conjecture at level $\Gamma_1(ND)$. We build on our recent work \cite{LW_compatibility} to get an unconditional congruence formula.

\item Step 2 (the computation of the $p$-Selmer group). A similar argument as in the case where $K$ is imaginary applies: one relates $\Sel_p(\tilde{J}^{(p)}_D/\Q)$ to $H^1(\Spec(A), \mu_p) \oplus \Hom(\Pic(A), \Z/p\Z) $ where $A = \mathcal{O}_K[1/N]$. If $K$ is real, then $A^{\times} \otimes \Z/p\Z$ has dimension $3$ (one dimension coming from the fundamental unit, and the two others from the prime dividing $N$ in $\mathcal{O}_K$), so the computation of $\Sel_p(\tilde{J}^{(p)}_D/\Q)$ is more complicated. A careful analysis of the Selmer conditions is necessary (we rely on the $N$-adic uniformization of $J_0(N)$ and results of \cite{deShalit} for the local condition at $N$).
\end{itemize}

Let us conclude this introduction by a few comments. 

The general philosophy behind our result is a relation between the analytic or algebraic aspects of $\GL_2$ and the ones of $\GL_1$. That is, we relate a degree $2$ $L$-function (the $L$-function of a cuspidal modular form) to a degree $1$ $L$-function (the $L$-function of a Dirichlet character). Similarly, we relate a Selmer group for an abelian variety to a Selmer group for a character (a class group). The reason behind such a relation is an Eisenstein congruence, between a cusp form and an Eisenstein series (as explained briefly in Step 1 of Mazur's result). This philosophy is of course not new, and dates back at least to the work of Ribet on the converse of Herbrand's theorem \cite{Ribet_Herbrand}. For a recent example, see \cite{Shankar}. One general phenomenon is that ``half'' of the $L$-values of the Eisenstein series are zero (here meaning that for $D>0$ we have $L(E_2, \chi_D, 1)=0$), and so direct Eisenstein congruences do not give much information. To go beyond, one needs a ``higher'' theory of Eisenstein congruences, \eg as developed in \cite{Lecouturier_higher}.

We expect that for a positive proportion and $< 100 \%$ of the $D$'s as in Theorem \ref{main_thm_bis}, we have $\Sel_p(\tilde{J}^{(p)}_D/\Q) \neq 0$. By Theorem \ref{main_thm_bis}, this is equivalent to saying that $u(K)^{h(K)}$ is a $p$th power modulo (a prime above) $N$ for a positive proportion and $< 100\%$ of the real quadratic fields $K$ in which $N$ splits (here, $N$ and $p$ are fixed). We do not know if classical techniques from number theory can yield results on this conjecture.

Let us notice that the congruence formula for the $L$-value modulo $p^2$ we obtain here can be proved using a different approach. Namely, one can use a formula of Popa (based on a general formula of Waldspurger) to express $L(f/K, 1)$, where $K$ is quadratic real and $f$ is a cuspidal eigenform, in terms of closed Heegner geodesics. We refer to \cite{Lecouturier_HV} for the details. We like our present approach because it does not rely on the difficult result of Waldspurger and Popa.

Finally, let us comment on the recent work of Castella--Grossi--Skinner \cite{Skinner_BSD}. They prove the $p$-part of the BSD formula for certain elliptic curves $E$ over $\Q$ of (algebraic) rank $\leq 1$ and with a degree $p$-isogeny over $\Q$ (\cf \cite[Theorem D]{Skinner_BSD}). Their method is to prove and rely on the Iwasawa main conjecture in this Eisenstein case \cite[Theorem A]{Skinner_BSD}. To do so, they actually prove Perrin--Riou's anti-cyclotomic Iwasawa main conjecture under some asumptions for an auxiliary imaginary quadratic field (\cf \cite[Conjecture B and Theorem C]{Skinner_BSD}). 

The results of \cite{Skinner_BSD} and ours, which are independent and use totally different methods, intersect in the following single example: when $E=X_0(11)$ and $p=5$. They prove the $p$-part of the BSD formula for quadratic twists $E_D$ of $E$ by $K=\Q(\sqrt{D})$ where $D>0$ is such that $p$ is \emph{inert} in $K$ (and $N=11$ splits in $K$). As mentioned, they also need to assume that the rank of $E_D$ is zero. We do not have such a restriction on the rank (which could be $\geq 2$ for us), and we also handle the case when $p$ splits in $K$. On the one hand, their result in this example and under their assumptions is more precise since they prove the exact $p$-part of BSD while we only prove BSD ``modulo $p$''. On the other hand,  our main result Theorem \ref{main_thm_bis} is more explicit since we give a criterion in terms of the fundamental unit and the class number. We also handle higher dimensional abelian varieties, which is not the case in \cite{Skinner_BSD} (we do not know whether they may extend their methods to modular eigenforms with non-rational Hecke fields).

The structure of our paper is as follows. In section \ref{section_periods}, we relate the various periods occurring in this paper. In section \ref{section_L_modSymb}, we reduce the question of the divisibility of $\Lambda(U_{\mathfrak{P}}, \chi_D)$ and $\frac{L(\tilde{J}^{(p)}_{D},1)}{\Omega_D}$ to a question about modular symbols. In section \ref{section_Sharifi}, we prove the result one needs on the conjecture of Sharifi. In section \ref{section_L_value}, we use the results of section \ref{section_Sharifi} to prove the modulo $p^2$ Eisenstein congruence of Step $1$ above. In section \ref{section_uniformization}, we recall some results of de Shalit regarding the $N$-adic uniformization of $J_0(N)$. Finally, in section \ref{section_selmer group computation}, we compute the completion at $\mathfrak{P}$ of the Selmer group of $\tilde{J}^{(p)}_D$.

\section{Notation}\label{section_notation}
We fix primes $N$ and $p$ $\geq 5$ such that $p \mid \mid N-1$ (\ie $\gcd(N-1,p^2)=p$). Fix a surjective group homomorphism $\log : (\Z/N\Z)^{\times} \rightarrow \Z/p\Z$.

We denote by $X_0(N)$ (resp. $X_1(N)$) the compact modular curve of level $\Gamma_0(N)$ (resp. $\Gamma_1(N)$). We denote by $H_1(X_0(N), \Z)$ the first singular homology group of $X_0(N)$ with coefficients in $\Z$. A similar notation applies to $X_1(N)$.

Fix algebraic closures $\overline{\Q}$ and $\overline{\Q}_p$ of $\Q$ and $\Q_p$ respectively, together with embeddings $\overline{\Q} \hookrightarrow \overline{\Q}_p$ and $\overline{\Q} \hookrightarrow \C$. If $x,y \in \C^{\times}$, we write $x \sim y$ if $\frac{x}{y}$ is an algebraic number which is also a $p$-adic unit with respect to our fixed embeddings. 

We fix a real quadratic field $K$ of discriminant $D>0$ such that $\gcd(D, Np)=1$. We denote by $\chi_D : (\Z/D\Z)^{\times} \rightarrow \{\pm 1\}$ the Dirichlet character corresponding to $K$. We denote by $u(K)$ and $h(K)$ a fundamental unit and the class number of $K$ respectively. We assume that $N$ splits in $K$.

\section{Periods} \label{section_periods}

Let us recall the definition of the real N\'eron period $\Omega_A$ of an abelian variety $A$ over $\Q$. Let $\mathcal{A}$ be the N\'eron model of $A$ over $\Z$ and $\omega_A$ be a generator of the global differential $g$-forms of $\mathcal{A}$, where $g=\dim(A)$. We define $\Omega_A = \int_{A(\mathbf{R})} \mid \omega_A \mid$.

\begin{prop}\label{prop_neron_period_twist}
Let $A_D$ denote the quadratic twist of $A$ by $K=\Q(\sqrt{D})$ (where $D$ is as above). Then $\Omega_A \sim \Omega_{A_D}$.
\end{prop}
\begin{proof}
Let $A':=A \times_{\Q} K$ and $A'_D = A_D \times_{\Q} K$. Let $\mathcal{A}'$ (resp. $\mathcal{A}'_D$) be the N\'eron model of $A'$ (resp. $A'_D$) over $\mathcal{O}_K[\frac{1}{D}]$. Since $A' \simeq A'_D$ over $K$, we have $\mathcal{A}' \simeq \mathcal{A}'_D$ over $\mathcal{O}_K[\frac{1}{D}]$. Since $\Spec(\mathcal{O}_K[\frac{1}{D}]) \rightarrow \Spec(\Z[\frac{1}{D}])$ is unramified, we have $\mathcal{A}' = \mathcal{A} \times \Spec(\mathcal{O}_K[\frac{1}{D}])$ and $\mathcal{A}'_D = \mathcal{A}_D \times \Spec(\mathcal{O}_K[\frac{1}{D}])$ where $\mathcal{A}$ (resp. $\mathcal{A}_D$) is the N\'eron model of $A$ (resp. $A_D$) over $\Z[\frac{1}{D}]$. Note that we have $A'(\mathbf{R}) = A(\mathbf{R}) \times A(\mathbf{R}) $. Since a generator of the global differential $g$-forms of $\mathcal{A}$ induce by base change a generator of the global differential $g$-forms of $\mathcal{A}'$, we get $\Omega_{A'} = \Omega_{A}^2$ (both quantity are actually only well-defined up to $D$-units, but in any case $\Omega_{A'} \sim \Omega_{A}^2$ since $p\nmid D$). Similarly, we get $\Omega_{A'} \sim \Omega_{A_D}^2$. We conclude that $\Omega_{A_D} \sim \Omega_A$.
\end{proof}

\begin{rems}
\begin{enumerate}
\item It would be interesting to have a precise formula relating $\Omega_A$ and $\Omega_{A_D}$. Such a formula has been found when $A$ is an elliptic curve in \cite{Pal_quadratictwist}. 
\item As in the introduction, we shall use the notation $\Omega_D$ instead of $\Omega_{A_D}$ for simplicity.
\end{enumerate}
\end{rems}

Let us now make the link between periods of modular abelian varieties and periods of modular eigenforms. If $f \in S_2(\Gamma_0(N))$ is a normalized eigenform, there exists $\Omega_f^+, \Omega_f^-$ in $\mathbf{C}$ such that for any $c \in H_1(X_0(N), \Z)^{\pm}$, the complex number $\frac{\int_c 2i\pi f(z)dz}{\Omega_f^{\pm}}$ is algebraic, and furthermore is $p$-adically integral under our fixed embeddings. Here, $H_1(X_0(N), \Z)^{\pm}$ is the subgroup of $H_1(X_0(N), \Z)$ on which the complex conjugation acts by $\pm 1$. We additionally require that there exists $c \in H_1(X_0(N), \Z)^{\pm}$ such that $\frac{\int_c 2i\pi f(z)dz}{\Omega_f^{\pm}}$ is a $p$-adic unit. Under these conditions, the periods $\Omega_f^{\pm}$ is unique up to a $p$-adic unit.

\begin{prop}\label{prop_product_periods}
We have $\Omega_{\tilde{J}^{(p)}} \sim \prod_{f \in S_2^{\Eis}(N,p)} \Omega_f^+$ (recall from Section \ref{Section_introduction} that $S_2^{\Eis}(N,p)$ is the set of eigenforms in $S_2(\Gamma_0(N))$ which are congruent to the Eisenstein series $E_2$ modulo some prime ideal above $p$).
\end{prop}
\begin{proof}
For simplicity, let us denote $\tilde{J}^{(p)}$ by $A$ and $\gamma_{\mathfrak{P}}$ by $\mathfrak{a}$ (\cf (\ref{eq_gamma_P}) for the definition of $\gamma_{\mathfrak{P}}$). We also let $T = (\mathbb{T}^0/\mathfrak{a}) \otimes_{\Z} \Z_p$. Let $$g = \dim(A) = \# S_2^{\Eis}(N,p)$$ and let $\mathcal{A}$ be the N\'eron model of $A$ over $\Z$. The pull-back map $$H^0(\mathcal{A}, \Omega_{\mathcal{A}}^1) \hookrightarrow H^0(J_0(N)/\Q, \Omega_{J_0(N)/\Q}^1) = S_2(\Gamma_0(N), \Q)$$ has image contained $S_2(\Gamma_0(N), \Z)[\mathfrak{a}]$, where for a ring $R \subset \mathbf{C}$, we denote $S_2(\Gamma_0(N), R)$ the subspace of cusp forms whose $q$-expansion at $\infty$ has coefficients in $R$. The \emph{Manin constant} of $A$ is by definition the index $$c_A:= [S_2(\Gamma_0(N), \Z)[\mathfrak{a}] : H^0(\mathcal{A}, \Omega_{\mathcal{A}}^1)]$$ (\cf \cite[\S 3.1]{Manin_constant}). By \cite[Theorem 3.5]{Manin_constant}, we have $p \nmid c_A$.

Let $$\phi : H_1(X_0(N), \Z)^+/\mathfrak{a}  \otimes \mathbf{C} \rightarrow \Hom(S_2(\Gamma_0(N), \Z)[\mathfrak{a}], \mathbf{C})$$ be the map given by the pairing $(c,f)\mapsto \int_c 2i\pi f(z)dz$ (where the `$+$' in the exponent means the invariants under the complex conjugation). By \cite[Proposition II.18.3]{Mazur_Eisenstein}, $H_1(X_0(N), \Z)^+/\mathfrak{a} \otimes \Z_p$ is free of rank $1$ over $T$. Let $c\in H_1(X_0(N), \Z)^+$ giving a generator of $H_1(X_0(N), \Z)^+/\mathfrak{a} \otimes \Z_p$ over $T$. Let $n_1, ..., n_g \in \mathbf{N}$ such that the Hecke operators $T_{n_1}$, ..., $T_{n_g}$ for a $\Z_p$-basis of $T$. Finally, fix a family $f_1$, ... $f_g$ in $S_2(\Gamma_0(N), \Z)[\mathfrak{a}]$ which gives a $\Z_p$-basis of $S_2(\Gamma_0(N), \Z_p)[\mathfrak{a}]$ (the latter being a free $T$-module of rank one). 

Since $p \nmid c_A$, we easily see that $\Omega_A \sim \det(M)$, where $$M = (\int_{T_{n_a}c} 2i\pi f_b(z)dz)_{1 \leq a,b \leq g} \text{.}$$ Let $F_1$, ..., $F_g$ be the elements of $S_2^{\Eis}(N,p)$. Note that $(F_1, ..., F_g)$ is a $\overline{\Q}_p$-basis of $S_2(\Gamma_0(N), \overline{\Q}_p)[\mathfrak{a}]$. Let $P$ be the matrix of $(F_1, ..., F_g)$ in $(f_1, ..., f_g)$. If $f \in S_2(\Gamma_0(N))$ and $n\in \mathbf{N}$, we denote by $a_n(f)$ the $n$th Fourier coefficient of $f$ at $\infty$. We have $$\det(a_{n_a}(F_b))_{1 \leq a,b \leq g} = \det(P)\cdot  \det(a_{n_a}(f_b))_{1 \leq a,b \leq g} \sim \det(P)$$ since $\det(a_{n_a}(f_b))_{1 \leq a,b \leq g}$ is a $p$-adic unit by construction. Thus, we have 
\begin{align*}
\det(M) &= \det(P)^{-1}\cdot \det(\int_{T_{n_a}c} 2i\pi F_b(z)dz)_{1 \leq a,b \leq g} \\ &=\det(P)^{-1} \cdot \det(\int_{c} 2i\pi a_{n_a}(F_b)\cdot F_b(z)dz)_{1 \leq a,b \leq g} \\& \sim \det(a_{n_a}(F_b))_{1 \leq a,b \leq g}^{-1}\cdot \det(\int_{c} 2i\pi a_{n_a}(F_b)\cdot F_b(z)dz)_{1 \leq a,b \leq g} \text{ .}
\end{align*}
By construction of $c$, we can choose $\Omega_{F_b}^+ = \int_c 2i\pi F_b(z)dz$ . We thus get 
\begin{align*}
\det(M) &\sim \det(a_{n_a}(F_b))_{1 \leq a,b \leq g}^{-1}\cdot \prod_{b=1}^g \Omega_{F_b}^+ \cdot \det(a_{n_a}(F_b))_{1 \leq a,b \leq g} 
\\& \sim \prod_{b=1}^g \Omega_{F_b}^+ \text{ .}
\end{align*}
This concludes the proof of Proposition \ref{prop_product_periods}.
\end{proof}

\section{$L$-values and modular symbols}\label{section_L_modSymb}
In this section, we reformulate the problem of computing the $p$-adic valuation of $\frac{L(\tilde{J}^{(p)}_{D},1)}{\Omega_D}$ in terms of modular symbols. 

Let us first introduce some basic notation regarding modular symbols. If $\Gamma$ is a finite index subgroup of $\SL_2(\Z)$, we denote by $X(\Gamma)$ the compact modular curve of level $\Gamma$ and by $C(\Gamma)$ the subset of cusps of $X(\Gamma)$.  One can consider the relative singular homology group $H_1(X(\Gamma), C(\Gamma), \Z)$. If $\alpha, \beta \in \mathbf{P}^1(\Q)$, we denote by $\{\alpha, \beta\}$ the class in $H_1(X(\Gamma), C(\Gamma), \Z)$ of the hyperbolic path in the extended upper-half plane from $\alpha$ to $\beta$. The element $\{\alpha, \beta\}$ is called a \emph{modular symbol}.

Recall that in section \ref{section_periods} we have defined a period $\Omega_f^+$ attached to a normalized eigenform $f \in S_2(\Gamma_0(N))$. Let $\mathfrak{m}$ be the maximal ideal of residual characteristic $p$ corresponding to $f$. If $H_1(X_0(N), \Z)_{\mathfrak{m}}^+$ (the $\mathfrak{m}$-adic completion of $H_1(X_0(N), \Z)^+$) is free of rank one over $\mathbb{T}_{\mathfrak{m}}^0$, then one may choose a generator $c$ of $H_1(X_0(N), \Z)_{\mathfrak{m}}^+$ and define $\Omega_f^+ = \int_c 2i\pi f(z)dz$. In this situation, the period map $\varphi_f : H_1(X_0(N), \Z)^+ \rightarrow \overline{\Q}_p$ given by $\gamma \mapsto \frac{\int_\gamma 2i\pi f(z)dz}{\Omega_f^+}$ takes values in $\overline{\Z}_p$. Furthermore, if $\sigma \in \Gal(\overline{\Q}_p/\Q_p)$ then we may normalize the periods so that $\varphi_{f^{\sigma}} = \sigma \circ \varphi_f$, where $f^{\sigma}$ is the conjugate of $f$ by $\sigma$. 

This discussion applies when $f \in S_2^{\Eis}(N,p)$. We fix such a $f$ in this section. The assumption that $p \mid \mid N-1$ implies that all the other such eigenforms are of the form $f^{\sigma}$ for some $\sigma \in \Gal(\overline{\Q}/\Q)$. Indeed, in this case the Eisenstein ideal $I$ is maximal \cite[Proposition II.9.7]{Mazur_Eisenstein}, and since it is principal \cite[Proposition II.16.6]{Mazur_Eisenstein} then we conclude that $\mathbb{T}_{\mathfrak{P}}^0$ is a discrete valuation ring, and in particular has no zero-divisors. 

Let us now give a reformulation of Hypothesis \ref{hypo}. We have just noticed that $\# S_2^{\Eis}(N,p)$ is the number of conjugates of $f$ by automorphisms in $\Gal(\overline{\Q}/\Q)$. On the other hand, $g_p$ is the number of conjugates of $f$ by automorphisms in $\Gal(\overline{\Q}_p/\Q_p)$ (recall that we have fixed an embedding $\overline{\Q} \hookrightarrow \overline{\Q}_p$). Therefore, Hypothesis \ref{hypo} holds if and only if all the conjugates of $f$ (\ie eigenforms in $S_2^{\Eis}(N,p)$) are actually conjugates of $f$ by some $\sigma \in \Gal(\overline{\Q}_p/\Q_p)$.

As is well-known (\cf \cite[\S 6 Proposition]{Mazur_1979}), if $\chi : (\Z/m\Z)^{\times} \rightarrow \mathbf{C}^{\times}$ is a primitive Dirichlet character of conductor $m>1$ coprime to $N$ then we have 
\begin{equation}\label{Birch_formula}
\tau(\chi^{-1})\cdot L(f, \chi, 1) = \sum_{a \in (\Z/m\Z)^{\times}} \chi(a)^{-1}\cdot \int_{\frac{a}{m}}^{0} 2i\pi f(z)dz
\end{equation}
where $\tau(\chi^{-1}) = \sum_{a=0}^{m-1} \chi(a)^{-1}\cdot e^{\frac{2i\pi a}{m}}$ is the Gauss sum attached to $\chi^{-1}$. Since $\gcd(m, N)=1$, the cusps $0$ and $\frac{a}{m}$ of $X_0(N)$ are equivalent, so the modular symbol $\{0, \frac{a}{m}\}$ is in $H_1(X_0(N), \Z)$. Therefore, if $\chi$ is even (\ie $\chi(-1)=1$), the formula (\ref{Birch_formula}) can be rewritten as
\begin{equation}\label{Birch_formula_bis}
\tau(\chi^{-1})\cdot \frac{L(f, \chi, 1)}{\Omega_f^+} = -\sum_{a \in (\Z/m\Z)^{\times}} \chi(a)^{-1}\cdot \varphi_f(\{0, \frac{a}{m}\}) \text{ .}
\end{equation}

We apply this formula for $\chi = \chi_D$. Note that $\tau(\chi)^2 = D$, so $\tau(\chi)$ is coprime to $p$. Following Mazur \cite[II \S 6 and III \S 2]{Mazur_1979}, we let
\begin{equation}\label{eq_def_univ_L}
\Lambda(U_{\mathfrak{P}}, \chi_D) = \sum_{a \in (\Z/D\Z)^{\times}} \chi_D(a) \cdot \{0, \frac{a}{D}\} \in H_1(X_0(N), \Z)_{\mathfrak{P}}^{+} \text{ .}
\end{equation}
In view of (\ref{Birch_formula_bis}), it is natural to call $\Lambda(U_{\mathfrak{P}}, \chi_D)$ the $\mathfrak{P}$-part of the universal special value twisted by $\chi_D$. We let
$$\Theta_D =  \sum_{a \in (\Z/D\Z)^{\times}} \chi_D(a) \cdot \{0, \frac{a}{D}\} \in H_1(X_0(N), \Z)^{+} \text{ .}$$
Thus, $\Lambda(U_{\mathfrak{P}}, \chi_D)$ is just the image of $\Theta_D$ into the completion at $\mathfrak{P}$ of $H_1(X_0(N), \Z)^{\pm}$. 

We have the following classical $L$-function decomposition
\begin{equation}\label{L_function_dec_eq}
L(\tilde{J}^{(p)}_{D}, s) = L(\tilde{J}^{(p)}, \chi_D, s) = \prod_{\sigma \in \Gal(\overline{\Q}/\Q)} L(f^{\sigma}, \chi_D, s) \text{ .}
\end{equation}
Combining Proposition \ref{prop_neron_period_twist}, Proposition \ref{prop_product_periods}, (\ref{Birch_formula_bis}) and (\ref{L_function_dec_eq}), we get that if Hypothesis \ref{hypo} is satisfied then
$$
\frac{L(\tilde{J}^{(p)}_{D}, 1)}{\Omega_D} \sim \prod_{\sigma \in \Gal(\overline{\Q}_p/\Q_p)} \sigma \left( \sum_{a \in (\Z/D\Z)^{\times}} \chi_D(a)\cdot \varphi_f(\{0, \frac{a}{D}\}) \right) \text{ .}
$$

Note that $\varphi_f(\{0, \frac{a}{D}\})$ takes values in a totally ramified finite extension of $\Z_p$ (indeed, $\mathbb{T}_{\mathfrak{P}}^0$ is a totally ramified DVR, \cf \cite[Proposition II.19.1]{Mazur_Eisenstein}). 
Thus, the $\mathfrak{P}$-adic valuation of $ \sum_{a \in (\Z/D\Z)^{\times}} \chi_D(a)\cdot \varphi_f(\{0, \frac{a}{D}\})$ is the same as the $p$-adic valuation of $\frac{L(\tilde{J}^{(p)}_{D}, 1)}{\Omega_D} $. This proves the following result.
\begin{prop}\label{prop_key_equivalence}
Assume Hypothesis \ref{hypo}. Then the $p$-adic valuation of the rational number $\frac{L(\tilde{J}^{(p)}_{D},1)}{\Omega_D}$ is the largest integer $n\geq 0$ such that 
$$ \Theta_D \in I^n \cdot H_1(X_0(N), \Z_p)^+  \text{ .}$$
\end{prop}
Proposition \ref{prop_key_equivalence} shows that Corollaries \ref{main_cor} and \ref{main_cor_bis} follows from Theorems \ref{main_thm} and \ref{main_thm_bis} respectively.

We have the following ``trivial divisibility'' due to Mazur when $D>0$.
\begin{prop}\label{prop_valuation_1}
We have $\Theta_D \in I \cdot H_1(X_0(N), \Z_p)^+$. Therefore, assuming Hypothesis \ref{hypo}, $p$ divides the numerator of $\frac{L(\tilde{J}^{(p)}_{D},1)}{\Omega_D}$.
\end{prop}
\begin{proof}
Mazur proved in \cite[Proposition II.18.8]{Mazur_Eisenstein} that we have an isomorphism $$\alpha : H_1(X_0(N), \Z_p)^+/I\cdot H_1(X_0(N), \Z_p)^+ \xrightarrow{\sim} \Z/p\Z$$ sending a modular symbol $\{0, \frac{b}{d}\}$ (with $\gcd(b,d)=\gcd(d,N)=1$) to $\log(d)$ (\cf Section \ref{section_notation} for the definition of $\log$). We have 
$$\alpha(\Theta_D) =\sum_{a \in (\Z/m\Z)^{\times}} \chi_D(a)\cdot \log(D) = 0 $$
since $\chi_D$ is non-trivial. This proves that $\Theta_D \in I \cdot H_1(X_0(N), \Z_p)^+$. The assertion about $\frac{L(\tilde{J}^{(p)}_{D},1)}{\Omega_D}$ follows from Proposition \ref{prop_key_equivalence}.
\end{proof}

We have thus proved the first assertion of Theorem \ref{main_thm} and Corollary \ref{main_cor}. In order to go further and study the divisibility of $\Theta_D$ by $I^2$, one will use results on a conjecture of Sharifi.

\section{Level compatibility in Sharifi's conjecture}\label{section_Sharifi} 

Let us recall some background on modular symbols and Sharifi's conjecture. For $M \in \mathbf{N}$, let $\Gamma_1(M) = \{\begin{pmatrix} a & b \\ c & d \end{pmatrix} \in \SL_2(\Z) \text{ such that } a-1 \equiv c \equiv 0 \text{ (modulo } M \text{)} \}$ and let $X_1(M)$ be the compact modular curve of level $\Gamma_1(M)$. 

Let $C_M = \Gamma_1(M) \backslash \mathbf{P}^1(\Q)$ be the set of cusps of $X_1(M)$, and $C_M^0$ be those cusps in $C_M$ of the form $\Gamma_1(M) \cdot \frac{a}{b}$ with $\gcd(a,b)=1$ and $a \not\equiv 0  \text{ (modulo } M \text{)}$ (in the case $b=0$ we have the cusp $\Gamma_1(M)\cdot \infty$). 

Let $H_1(X_1(M), C_M, \Z)$ be the singular homology of $X_1(M)$ relative to $C_M$. If $\alpha$ and $\beta$ are in $\mathbb{P}^1(\Q)$, let $\{\alpha, \beta\}$ be the class in $H_1(X_1(M), C_M, \Z)$ of the hyperbolic geodesic from $\alpha$ to $\beta$ in $X_1(M)$.

Let $$\xi_M : \Z[\Gamma_1(M) \backslash \SL_2(\Z)] \rightarrow H_1(X_1(M), C_M, \Z)$$ be the (modified) Manin map: it sends a coset $\Gamma_1(M) \cdot \begin{pmatrix} a & b \\ c & d \end{pmatrix}$ to $\{-\frac{d}{Mb}, -\frac{c}{Ma}\}$ (it is the usual Manin map composed with the Atkin--Lehner involution $W_M$). Manin showed that $\xi_M$ is surjective. 

Let $S_M^0 \subset \Gamma_1(M) \backslash \SL_2(\Z)$ be the subset consisting of $\Gamma_1(M) \cdot \begin{pmatrix} a & b \\ c & d \end{pmatrix}$ with $M \nmid c$ and $M \nmid d$. Note that $S_M^0$ can be identified with the set of pairs $[c,d]$ where $c,d \in \Z/M\Z - \{0\}$ modulo the identification $[c,d]=[-c,-d]$.
The restriction $$\xi_M^0 : \Z[S_M^0] \rightarrow H_1(X_1(M), C_M, \Z)$$ is surjective (\cf \cite[\S 2.1.3]{Sharifi_survey}). 

If $A$ is a commutative ring, let $K_2(A)$ be the second K-group of $A$, as defined by Quillen. Let $\zeta_M \in \overline{\Q}$ be a primitive $M$th root of unity. There is an action of $\Gal(\Q(\zeta_M)/\Q)$ (and in particular of the complex conjugation) on $K_2(\Z[\zeta_M, \frac{1}{M}])$. We denote by $\mathcal{K}_M$ the largest quotient of $K_2(\Z[\zeta_M, \frac{1}{M}]) \otimes \Z[\frac{1}{2}]$ on which the complex conjugation acts trivially. The map $\Z[S_M^0] \rightarrow \mathcal{K}_M$ sending $\Gamma_1(M)  \cdot \begin{pmatrix} a & b \\ c & d \end{pmatrix}$ to the Steinberg symbol $\langle 1 - \zeta_M^c, 1-\zeta_M^d \rangle$ factors through $\xi_M^0$ (\cf \cite[\S 2.1.4]{Sharifi_survey}), and thus induces a map
$$\varpi_M : H_1(X_1(M), C_M^0, \Z) \rightarrow \mathcal{K}_M \text{ .}$$

Sharifi conjectured that $\varpi_M$ is annihilated by the Hecke operators $T_{\ell}-\ell\langle \ell \rangle - 1$ for primes $\ell$ not dividing $M$ (\cf the remark after Theorem 4.3.6 in \cite{Sharifi_Venkatesh}). This conjecture has a history of partial results: \cite{Fukaya_Kato}, \cite{LW} and most recently \cite{Sharifi_Venkatesh} and \cite{LW_compatibility}. In particular, the restriction of $\varpi_M$ to $H_1(X_1(M), \Z)$ is known to be annihilated by $T_{\ell}-\ell\langle \ell \rangle - 1$ for primes $\ell$ not dividing $M$ (\cf \cite[Theorem 4.3.7]{Sharifi_Venkatesh}, where we warn the reader that they use \emph{usual} Manin symbols and \emph{dual} Hecke operators). If $M = N$ is prime, then a mild improvement on the techniques of Sharifi and Venkatesh shows that $\varpi_N$ (not restricted) is annihilated by the Hecke operators $T_{\ell}-\ell\langle \ell \rangle - 1$ for primes $\ell \neq N$ \cite[Remark 1.1 (v)]{LW_compatibility}. 

We shall make use of the following result of \cite{LW}.
\begin{thm}\label{thm_main_thm_LW}
We have a commutative diagram
\begin{center}
\begin{tikzcd}
H_1(X_1(N), \Z_p)^+ \arrow[r, "\varpi_{N}"] \arrow[d, "\pi"] & J\cdot (\mathcal{K}_N \otimes \Z_p) \arrow[d] \\
I\cdot H_1(X_0(N), \Z_p)^+/I^2\cdot H_1(X_0(N), \Z_p) \arrow[r, "\sim"] & J\cdot (\mathcal{K}_N \otimes \Z_p)/J^2\cdot (\mathcal{K}_N \otimes \Z_p),
\end{tikzcd}
\end{center}
where $\pi$ is the forgetful map (induced by $z\mapsto z$ on the upper-half plane), $J$ is the augmentation ideal of $\Z[\Gal(\Q(\zeta_N)/\Q)]$, the vertical arrows are surjective and the lower horizontal map is an isomorphism. (The right vertical arrow is the obvious projection.)
\end{thm}
\begin{proof}
Let us explain how this result follows from \cite{LW}. By \cite[Proposition 2.3 (ii)]{LW}, the map $$\pi : H_1(X_1(N), \Z_p)^+ \rightarrow H_1(X_0(N), \Z_p)^+$$ has kernel $J \cdot H_1(X_1(N), \Z_p)^+$, where we identify $\Gal(\Q(\zeta_N)\Q)$ with $(\Z/N\Z)^{\times}$ (which acts by diamond operators on $H_1(X_1(N), \Z_p)^+$). 

By \cite[Proposition 2.12 (ii)]{LW} and \cite[Lemma 2.15]{LW}, the map $\varpi_N : H_1(X_1(N), \Z_p)^+ \rightarrow \mathcal{K}_N \otimes \Z_p$ takes values in $J\cdot (\mathcal{K}_N \otimes \Z_p)$. Finally, by \cite[Lemma II.18.7]{Mazur_Eisenstein}, the image of $\pi$ is $I \cdot H_1(X_0(N), \Z_p)^+$.

Thus, we have a commutative diagram 
\begin{center}
\begin{tikzcd}
H_1(X_1(N), \Z_p)^+ \arrow[r, "\varpi_{N}"] \arrow[d, "\pi"] & J\cdot (\mathcal{K}_N \otimes \Z_p) \arrow[d] \\
I\cdot H_1(X_0(N), \Z_p)^+ \arrow[r] & J\cdot (\mathcal{K}_N \otimes \Z_p)/J^2\cdot (\mathcal{K}_N \otimes \Z_p),
\end{tikzcd}
\end{center}
where the vertical arrows are surjective. 

The map $\varpi_{N} : H_1(X_1(N), \Z_p)^+ \rightarrow J\cdot (\mathcal{K}_N \otimes \Z_p)$ is surjective by the proof of \cite[Proposition 2.14 (b)]{LW}.  We know that the map $\varpi_N$ is annihilated by the Hecke operators $T_{\ell}-\ell\langle \ell \rangle - 1$ for primes $\ell \neq N$ and that $I$ is generated by the operators $T_{\ell}-\ell-1$ for $\ell \neq N$. 

Thus, the map $$I\cdot H_1(X_0(N), \Z_p)^+ \rightarrow J\cdot (\mathcal{K}_N \otimes \Z_p)/J^2\cdot (\mathcal{K}_N \otimes \Z_p)$$ factors through a surjective map $$I\cdot H_1(X_0(N), \Z_p)^+/I^2\cdot H_1(X_0(N), \Z_p)^+ \rightarrow J\cdot (\mathcal{K}_N \otimes \Z_p)/J^2\cdot (\mathcal{K}_N \otimes \Z_p) \text{ .}$$ By \cite[Remark 1.3]{LW}, we have a (canonical) isomorphism $$J\cdot (\mathcal{K}_N \otimes \Z_p)/J^2\cdot (\mathcal{K}_N \otimes \Z_p) \simeq ((\Z/N\Z)^{\times})^{\otimes 2} \otimes \Z_p \text{ .}$$ 

On the other hand, $I\cdot H_1(X_0(N), \Z_p)^+/I^2\cdot H_1(X_0(N), \Z_p)^+$ is also (canonically) isomorphic to $$J\cdot (\mathcal{K}_N \otimes \Z_p)/J^2\cdot (\mathcal{K}_N \otimes \Z_p) \simeq ((\Z/N\Z)^{\times})^{\otimes 2} \otimes \Z_p$$ (this is due to Mazur, see \cite[p. 2]{LW}). Therefore, the surjective map $$I\cdot H_1(X_0(N), \Z_p)^+/I^2\cdot H_1(X_0(N), \Z_p)^+ \rightarrow J\cdot (\mathcal{K}_N \otimes \Z_p)/J^2\cdot (\mathcal{K}_N \otimes \Z_p)$$ has to be an isomorphism. This concludes the proof of Theorem \ref{thm_main_thm_LW}.
\end{proof}

Another important aspect of Sharifi's theory is the way in which the maps $\varpi_M$ relate with each others when varying $M$. This has been studied under some assumptions in \cite{Fukaya_Kato} and \cite{Scott}. More recently, the authors \cite{LW_compatibility} extended these results using the novel techniques of Sharifi and Venkatesh. For the convenience of the reader, let us summarize the result of \cite{LW_compatibility} we shall need in the present paper.

\begin{thm}\cite[Theorem 1.4]{LW_compatibility} \label{thm_LW_compatibility}
Let $q\geq 2$ be a prime number and $M \geq 4$. Let $\pi_1, \pi_2 : X_1(Mq) \rightarrow X_1(M)$ be the two usual degeneracy maps, given respectively by $z\mapsto z$ and $z\mapsto qz$ on the upper-half plane. Let $C \subset C_{Mq}^0$ be a subset of cusps which are all in the same orbit under the action of $\Ker((\Z/Mq\Z)^{\times} \rightarrow (\Z/M\Z)^{\times})$ (the action being given by diamond operators). 

\begin{enumerate}
\item\label{main_thm_i} Assume that $q$ divides $M$. We have a commutative diagram
\begin{center}
\begin{tikzcd}
H_1(X_1(Mq), C, \Z) \arrow[r, "\varpi_{Mq}"] \arrow[d, "\pi_1"] & \mathcal{K}_{Mq} \arrow[d, "\Norm"] \\
H_1(X_1(M), \Z)\arrow[r, "\varpi_{M}"] &  \mathcal{K}_{M}.
\end{tikzcd}
\end{center}

\item\label{main_thm_ii} Assume that $q$ does not divide $M$. We have a commutative diagram
\begin{center}
\begin{tikzcd}
H_1(X_1(Mq), C, \Z) \arrow[r, "\varpi_{Mq}"] \arrow[d, "\pi_1 - \langle p \rangle \pi_2"] & \mathcal{K}_{Mq} \arrow[d, "\Norm"] \\
H_1(X_1(M), \Z)\arrow[r, "\varpi_{M}"] &  \mathcal{K}_{M}.
\end{tikzcd}
\end{center}
Here, $\langle q \rangle$ is the $q$th diamond operator, induced by the action of a matrix $\begin{pmatrix} a & b \\ c & d \end{pmatrix} \in \Gamma_0(M)$ with $d \equiv q \text{ (modulo }M\text{)}$ on $X_1(M)$.
\end{enumerate}

\end{thm}

\section{Congruence formula for the $L$-value}\label{section_L_value}

We now have all the tools we need to study the divisibility of $\Theta_D$ by $I^2$, and thus prove the equivalence between (\ref{main_thm_bis_1}) and (\ref{main_thm_bis_2}) in Theorem \ref{main_thm_bis} (and similarly for Corollary \ref{main_cor_bis} by Proposition \ref{prop_key_equivalence}). 

Recall that we have let  $$\Theta_D =  \sum_{a \in (\Z/D\Z)^{\times}} \chi_D(a)\cdot \{0, \frac{a}{D}\} \in H_1(X_0(N), \Z_p)^+$$ and that $\Theta_D \in I\cdot H_1(X_0(N), \Z_p)^+$. By Theorem \ref{thm_main_thm_LW}, we have $\Theta_D \in I^2\cdot H_1(X_0(N), \Z_p)^+$ if and only if $\varpi_N(\tilde{\Theta}_D) \in J^2\cdot  (\mathcal{K}_N \otimes \Z_p)$, where $$\tilde{\Theta}_D = \sum_{a \in (\Z/D\Z)^{\times}} \chi_D(a)\cdot \{0, \frac{a}{D}\}  \in H_1(X_1(N), \Z_p)^+$$ (note that the boundary of $\tilde{\Theta}_D$ is indeed zero by as the cusps $\frac{a}{D}$ are all equivalent in $X_1(N)$).

We now face the problem that $\tilde{\Theta}_D$ is not in any explicit way a linear combination of Manin symbols. We thus do not get a direct explicit formula for $\varpi_N(\tilde{\Theta}_D)$. But this can be resolved if we move to level $\Gamma_1(ND)$. Let $$\tilde{\Theta}_D'=\sum_{a \in (\Z/D\Z)^{\times}} \chi_D(a)\cdot \{0, \frac{a}{D}\} \in H_1(X_1(DN), C_{ND}^0, \Z_p)^+ \text{ .}$$ Since $\chi_D$ is a primitive character modulo $D$, one easily sees that Theorem \ref{thm_LW_compatibility} yields $$\Norm(\varpi_{DN}(\tilde{\Theta}_D')) = \varpi_N(\tilde{\Theta}_D)\text{ ,}$$ where $\Norm : \mathcal{K}_{ND} \otimes \Z_p \rightarrow \mathcal{K}_N \otimes \Z_p$ is the norm map.

We now are in a better shape, since $\tilde{\Theta}_D' =\sum_{a \in (\Z/D\Z)^{\times}} \chi_D(a) \cdot \xi_{ND}^0([-Na, 1])$. We thus get
\begin{equation}\label{eq_norm_Theta}
\varpi_N(\tilde{\Theta}_D) = \Norm\left( \langle \prod_{a \in (\Z/D\Z)^{\times}} (1-\zeta_D^{a})^{\chi_D(a)}, 1-\zeta_{ND} \rangle\right) \in J\cdot \mathcal{K}_N \text{ .}
\end{equation}

\begin{prop}\label{prop_Norm_K2_unit}
We have 
$$\Norm\left( \langle \prod_{a \in (\Z/D\Z)^{\times}} (1-\zeta_D^{a})^{\chi_D(a)}, 1-\zeta_{ND} \rangle\right)  \in J^2\cdot  (\mathcal{K}_N \otimes \Z_p)$$
if and only if $u(K)^{h(K)}$ is a $p$th power modulo a prime above $N$ in $K$.
\end{prop}
\begin{proof}
Let $\iota : \mathcal{K}_N \otimes \Z_p \rightarrow \mathcal{K}_{ND} \otimes \Z_p$ be the functorial map coming from the inclusion $\Q(\zeta_N) \subset \Q(\zeta_{ND})$. If $x \in \mathcal{K}_{ND}$, then we have $$\iota \circ \Norm (x) = \sum_{g \in \Gal(\Q(\zeta_{ND})/\Q(\zeta_N))} g\cdot x \text{ .}$$ 

Write $\zeta_{ND} = \zeta_N^u \cdot \zeta_D^v$ for some $u, v \in \Z$ with $uD+vN=1$.
We thus get:
\begin{align*}
\iota \circ & \Norm \left( \langle \prod_{a \in (\Z/D\Z)^{\times}}  (1-\zeta_D^{a})^{\chi_D(a)}, 1-\zeta_{ND} \rangle\right) \\& = \sum_{b \in (\Z/D\Z)^{\times}} \langle \prod_{a \in (\Z/D\Z)^{\times}} (1-\zeta_D^{ab})^{\chi_D(a)}, 1-\zeta_N^u \zeta_D^{vb} \rangle 
\\& = \langle \prod_{a \in (\Z/D\Z)^{\times}} (1-\zeta_D^{a})^{\chi_D(a)}, \prod_{b \in (\Z/D\Z)^{\times}} (1-\zeta_N^u \zeta_D^{vb})^{\chi_D(b)} \rangle 
\\& =\chi_D(v)\cdot \langle \prod_{a \in (\Z/D\Z)^{\times}} (1-\zeta_D^{a})^{\chi_D(a)}, \prod_{b \in (\Z/D\Z)^{\times}} (1-\zeta_N^u \zeta_D^{b})^{\chi_D(b)} \rangle 
\end{align*}
We compute $\Norm_{\Q(\zeta_{ND})/\Q(\zeta_D)}(1-\zeta_N^u \zeta_D^{b}) = \frac{1-\zeta_D^{Nb}}{1-\zeta_D^b}$. Since $\chi_D(N)=1$ by assumption, we get  $$\Norm_{\Q(\zeta_{ND})/\Q(\zeta_D)}  \prod_{b \in (\Z/D\Z)^{\times}} (1-\zeta_N^u \zeta_D^{b})^{\chi_D(b)}  = 1 \text{ .}$$

Note also that  $$\prod_{b \in (\Z/D\Z)^{\times}} (1-\zeta_N^u \zeta_D^{b})^{\chi_D(b)}  \in K(\zeta_N)$$ and $$\prod_{a \in (\Z/D\Z)^{\times}} (1-\zeta_D^{a})^{\chi_D(a)} \in K \text{ .}$$ By Hilbert 90, there exists $y \in K(\zeta_N)^{\times}$ such that $$\prod_{b \in (\Z/D\Z)^{\times}} (1-\zeta_N^u \zeta_D^{b})^{\chi_D(b)}  = \frac{\sigma(y)}{y} \text{ ,}$$ where $\sigma \in \Gal(K(\zeta_N)/K) \simeq (\Z/N\Z)^{\times}$ is a fixed generator. We thus get
\begin{equation}\label{eq_norm_iota}
\iota \circ \Norm \left( \langle \prod_{a \in (\Z/D\Z)^{\times}}  (1-\zeta_D^{a})^{\chi_D(a)}, 1-\zeta_{ND} \rangle\right) = (\sigma -1)\cdot \alpha \text{ ,}
\end{equation}
where $\alpha \in K_2(K(\zeta_N^+)) \otimes \Z_p$ is given by $$\alpha = \chi_D(v)\cdot \langle \prod_{a \in (\Z/D\Z)^{\times}} (1-\zeta_D^{a})^{\chi_D(a)}, y \rangle \text{ .}$$ Here, $K(\zeta_N^+)$ is the totally real subfield of $K(\zeta_N)$ and by $\langle \prod_{a \in (\Z/D\Z)^{\times}} (1-\zeta_D^{a})^{\chi_D(a)}, y \rangle$ we mean the projection of the same Steinberg symbol on the fixed part by the complex conjugation. 

Such an $\alpha$ is unique up to an element in $(K_2(K(\zeta_N^+)) \otimes \Z_p) [J]$ (where we view $J$ as the augmentation ideal of $\Z[\Gal(K(\zeta_N)/K)]$).
\begin{lem}\label{lem_J_inv_residue}
Let $\partial : K_2(K(\zeta_{N}^+)) \otimes \Z_p \rightarrow \prod_{\mathfrak{q}} \mathbf{F}_{\mathfrak{q}}^{\times} \otimes \Z_p$ be the tame residue symbol map, where $\mathfrak{q}$ runs through prime divisors of $N$ in the ring of integers of $K(\zeta_{N}^+)$. Then the restriction of $\partial$ to $(K_2(K(\zeta_{N}^+)) \otimes \Z_p) [J]$ is the zero map.
\end{lem}
\begin{proof}
We claim that $(K_2(K(\zeta_{N}^+)) \otimes \Z_p) [J]$ is the image of $K_2(K) \otimes \Z_p$ in $K_2(K(\zeta_{N}^+)) \otimes \Z_p$ (via the functorial map). By \cite[Theorem 5.4]{Tate_K_2}, the \'etale Chern class map induces an injection $$K_2(K(\zeta_{N}^+)) \otimes \Z_p \hookrightarrow H^2(K(\zeta_{N}^+), \Z_p(2))$$ whose image is the torsion group of $H^2(K(\zeta_{N}^+), \Z_p(2))$. Actually, $H^2(K(\zeta_{N}^+), \Z_p(2))$ is a torsion group, since $H^2(\mathcal{O}_S, \Z_p(2))$ is finite and we have an exact sequence
$$0 \rightarrow H^2(\mathcal{O}_S, \Z_p(2)) \rightarrow H^2(\Q(\zeta_{N}^+), \Z_p(2)) \rightarrow \oplus_{\mathfrak{p} \not\in S} \mathbf{F}_{\mathfrak{p}}^{\times} \rightarrow 0\text{ ,}$$
where $\mathcal{O}_S$ is the ring of $S$-integers of $K(\zeta_{N}^+)$ for some finite set of places $S$. The Chern class map is therefore an isomorphism. Thus, to prove the above claim, it suffices to prove that the restriction map $$H^2(K, \Z_p(2)) \rightarrow H^2(K(\zeta_{N}^+), \Z_p(2))[J]$$ is surjective.  By \cite[Proposition 2.9]{Kolster}, it is enough to check that $$H^1(K(\zeta_{N}^+), \Z_p(2))=0 \text{ .}$$This follows from \cite[Remark 46]{Weibel}.

To conclude the proof of the Lemma, simply note that the following diagram is commutative:
\begin{center}
\begin{tikzcd}
K_2(K) \otimes \Z_p \arrow[r, "\partial"] \arrow[d] & \prod_{\mathfrak{q}' \mid N} \mathbf{F}_{\mathfrak{q}'}^{\times} \otimes \Z_p\arrow[d, "0"] \\
K_2(K(\zeta_{N}^+)) \otimes \Z_p \arrow[r, "\partial"] & \prod_{\mathfrak{q} \mid N} \mathbf{F}_{\mathfrak{q}}^{\times}  \otimes \Z_p.
\end{tikzcd}
\end{center}
This follows from the definition of the tame residue symbol and the fact that the ramification index of a prime $\mathfrak{q}' \mid N$ of $K$ in $K(\zeta_{N}^+)$ in divisible by $p$ (note that we have $\mathbf{F}_{\mathfrak{q}} = \mathbf{F}_{\mathfrak{q}'} = \mathbf{F}_N$ by assumption). 
\end{proof}

By Lemma \ref{lem_J_inv_residue}, all  $\alpha \in K_2(K(\zeta_N^+)) \otimes \Z_p$ satisfying (\ref{eq_norm_iota}) have the same image under the residue map in $\mathbf{F}_N^{\times} \otimes \Z_p$. We know that $$ \Norm \left( \langle \prod_{a \in (\Z/D\Z)^{\times}}  (1-\zeta_D^{a})^{\chi_D(a)}, 1-\zeta_{ND} \rangle\right) = (\sigma -1)\cdot x$$ for some $x \in K_2(\Z[\zeta_N, \frac{1}{N}])$. We have 
$$\Norm\left( \langle \prod_{a \in (\Z/D\Z)^{\times}} (1-\zeta_D^{a})^{\chi_D(a)}, 1-\zeta_{ND} \rangle\right)  \in J^2\cdot  (\mathcal{K}_N \otimes \Z_p)$$
if and only if $x \in J \cdot K_2(\Z[\zeta_N, \frac{1}{N}])$, if and only if the residue of $x$ in $\mathbf{F}_N^{\times} \otimes \Z_p$ is zero (\cf \cite[Proposition 2.12]{LW}). By the above, this happens if and only if the residue of $$\langle \prod_{a \in (\Z/D\Z)^{\times}} (1-\zeta_D^{a})^{\chi_D(a)}, y \rangle$$ is zero in $\mathbf{F}_N^{\times} \otimes \Z_p$. 

Notice that the class number formula for $K$ is equivalent to $$\abs{u(K)}^{ \pm h(K)} =  \prod_{a \in (\Z/D\Z)^{\times}} (1-\zeta_D^{a})^{\chi_D(a)}$$ (for some sign $\pm$ depending on the Gauss sum $\mathcal{G}(\chi_D) = \pm \sqrt{D}$). Therefore, the residue of $$\langle \prod_{a \in (\Z/D\Z)^{\times}} (1-\zeta_D^{a})^{\chi_D(a)}, y \rangle$$ at a prime $\mathfrak{N}$ above $N$ in $K$ is $$(u(K)^{ \pm h(K)})^{v_{\mathfrak{N}}(y)} \text{ modulo }\mathfrak{N} \text{ ,}$$ where $v_{\mathfrak{N}}(y)$ is the $\mathfrak{N}$-adic valuation of $y$. 

Let $K(\zeta_N)_{\mathfrak{N}}$ be the $\mathfrak{N}$-adic completion of $K(\zeta_N)$. The element $1-\zeta_N$ is a uniformizer at $\mathfrak{N}$, so one may write $$y = u \cdot (1-\zeta_N)^{v_{\mathfrak{N}}(y)}$$ for some local unit $u$. 

Say our generator $\sigma$ of $\Gal(\Q(\zeta_N)/\Q)$ corresponds to a generator $g$ of $(\Z/N\Z)^{\times}$. We then have 
$$\frac{\sigma(y)}{y} \equiv \left(\frac{1-\zeta_N^g}{1-\zeta_N}\right)^{v_{\mathfrak{N}}(y)} \equiv g^{v_{\mathfrak{N}}(y)} \text{ (modulo } \mathfrak{N} \text{).}$$
On the other hand, by definition we have 
$$\frac{\sigma(y)}{y} \equiv  \prod_{a \in (\Z/D\Z)^{\times}} (1-\zeta_D^{a})^{\chi_D(a)} \equiv u(K)^{ \pm h(K)}  \text{ (modulo } \mathfrak{N} \text{).}$$
Therefore, $p$ divides $v_{\mathfrak{N}}(y)$ if and only if $u(K)^{ \pm h(K)} $ is a $p$th power modulo $\mathfrak{N}$. This concludes the proof of Proposition \ref{prop_Norm_K2_unit}.
\end{proof}

In conclusion, we have proven the equivalence of (\ref{main_thm_bis_1}) and (\ref{main_thm_bis_2}) in Theorem \ref{main_thm_bis}.

\section{The $N$-adic uniformization of $J_0(N)$}\label{section_uniformization}

The goal of this section is to prove the following result, which will be crucial to understand the image of the local Kummer maps defining our Selmer group.
Our proof relies on the explicit description of $\tilde{J}^{(p)}(\Q_N)$ based on results of Ehud de Shalit \cite{deShalit}. 

\begin{prop}\label{thm_uniformization}
We have a group isomorphism
$$\left(\tilde{J}^{(p)}(\Q_N)/I\cdot \tilde{J}^{(p)}(\Q_N)\right) \otimes \Z_p \simeq (\Z/p\Z)^2 \text{ .}$$
\end{prop}
\begin{proof}
Let us recall some basic results regarding the $N$-adic uniformization of $J_0(N)$. We shall use the results of \cite{Lecouturier_MT}, but we emphasise that the results we shall use follow easily from de Shalit's work \cite{deShalit}. Let us note that our $N$ and $p$ correspond to $p$ and $\ell$ respectively in \cite{Lecouturier_MT}.

Let $S$ be the (finite) set of supersingular points in characteristic $N$ of $X_0(N)$, \ie the set of supersingular elliptic curves over $\overline{\mathbf{F}}_N$ up to isomorphism. Let $N_0 = \Z[S]^0$ be the subgroup of degree zero elements in $\Z[S]$. There is a natural action of $\mathbb{T}^0$ on $N_0$ (where we recall that $\mathbb{T}^0$ is the cuspidal Hecke algebra of weight $2$ and level $\Gamma_0(N)$ over $\Z$). 

We have a $\mathbb{T}^0$-equivariant isomorphism
$$
J_0(\Q_N) \simeq \Hom(N_0, \Q_N^{\times})/q_0(N_0)\text{ ,}
$$
where $$q_0 : N_0 \rightarrow \Hom(N_0, \Q_N^{\times})$$ is a certain injective homomorphism (\cf \cite[\S 1.3]{Lecouturier_MT}). 

We shall need the following
\begin{fact}
$$q_0(N_0) \subset I \cdot \Hom(N_0, \Q_{N^2}^{\times})  \text{ .}$$ 
\end{fact}
\begin{proof} (Fact) In \cite{deShalit_pairing} (\cf also \cite[\S 1.3]{Lecouturier_MT}), de Shalit proved that there is a homomorphism $q : \Z[S] \rightarrow \Hom(\Z[S], \Q_{N^2}^{\times})$ inducing $q_0$. By \cite[Proposition 1.3 (ii)]{Lecouturier_MT}, $q$ is $\mathbb{T}$-equivariant (where $\mathbb{T}$ is the full Hecke algebra acting faithfully on $\Z[S]$). Thus, the fact follows from the equality
\begin{equation}\label{eq_N_0_I_N}
N_0 = I\cdot \Z[S] \text{ .}
\end{equation}
We obviously have $I \cdot \Z[S] \subset N_0$ (\cf \cite[Theorem 3.1 (i)]{Emerton}). Since $\Z[S]/N_0 = \Z$ (via the degree map), (\ref{eq_N_0_I_N}) follows if one can show that $\Z[S]/I\cdot \Z[S] \simeq \Z$ (the right hand side has an action of $\mathbb{T}$ via $\mathbb{T}/I \simeq \Z$). It suffices to check this after completion at every maximal ideal $\mathfrak{m}$ of $\mathbb{T}$. If $\mathfrak{m}$ is not Eisenstein, \ie if $I$ is not contained in $\mathfrak{m}$, this is clear. If $\mathfrak{m}$ is Eisenstein, this follows from the fact that $\Z[S]$ is locally free at $\mathfrak{m}$ \cite[Theorem 4.2]{Emerton}.
\end{proof}

This proves that 
\begin{equation}\label{eq_J_0_unif}
\left(\tilde{J}^{(p)}(\Q_N)/I\cdot \tilde{J}^{(p)}(\Q_N)\right) \otimes \Z_p \simeq  \left( \Hom(N_0, \Q_N^{\times}) /I \cdot \Hom(N_0, \Q_N^{\times}) \right) \otimes \Z_p \text{ .} 
\end{equation}

By \cite[Corollary II.16.3]{Mazur_Eisenstein} and \cite[Theorem 0.5]{Emerton}, the $\mathbb{T}_{\mathfrak{P}}^0$-module $\Hom(N_0, \Z_p) \otimes_{\mathbb{T}}\mathbb{T}_{\mathfrak{P}}^0$ is free of rank one. Since
$$\left( \Hom(N_0, \Q_N^{\times}) /I \cdot \Hom(N_0, \Q_N^{\times}) \right) \otimes \Z_p \simeq   \Hom(N_0 , \Q_N^{\times}) \otimes_{\mathbb{T}^0} \mathbb{T}_{\mathfrak{P}}^0/I \simeq \Q_N^{\times} \otimes_{\Z} \mathbb{T}_{\mathfrak{P}}^0/I $$
and $\Q_N^{\times} \otimes_{\Z} \Z_p \simeq (\Z/p\Z)^2$, we get
$$\left( \Hom(N_0, \Q_N^{\times}) /I \cdot \Hom(N_0, \Q_N^{\times}) \right) \otimes \Z_p \simeq \mathbb{T}_{\mathfrak{P}}^0/I  \otimes_{\Z} (\Z/p\Z)^2  \simeq (\Z/p\Z)^2 \text{ .}$$
By (\ref{eq_J_0_unif}), we get 
$$\left(\tilde{J}^{(p)}(\Q_N)/I\cdot \tilde{J}^{(p)}(\Q_N)\right) \otimes \Z_p \simeq (\Z/p\Z)^2 $$
as wanted.
\end{proof}

\section{Selmer group computation}\label{section_selmer group computation}
The goal of this section is to prove the equivalence of (\ref{main_thm_bis_2}) and (\ref{main_thm_bis_3}) in Theorem \ref{main_thm_bis}. 

Notation in this section:  

We let $G_{\Q} = \Gal(\overline{\Q}/\Q)$ and $G_K = \Gal(\overline{\Q}/K)$. If $v$ is a place of $\Q$, we fix a decomposition group $G_{\Q_v} \subset G_{\Q}$ at $v$. Similarly, if $w$ is a place of $K$ above a place $v$ of $\Q$, we fix a decomposition group $G_{K_w} \subset G_K$ at $w$. Note that we do not necessarily have $G_{K_w} \subset G_{\Q_v}$, but there exists $g \in G_{\Q}$ such that $G_{K_w} \subset gG_{\Q_v}g^{-1}$. Actually, one can choose either $g=1$ or $g=g_0$, where $g_0 \in G_{\Q}$ is fixed so that the restriction of $g_0$ to $K$ generates $\Gal(K/\Q)$. We abuse notation and consider $G_{K_w}$ as a subgroup of $G_{\Q_v}$ via the isomorphism $gG_{\Q_v}g^{-1} \xrightarrow{\sim} G_{\Q_v}$.

We denote by $\chi_p : G_{\Q} \rightarrow \Z_p^{\times}$ the $p$-adic cyclotomic character and we let $\overline{\chi}_p : G_{\Q} \rightarrow (\Z/p\Z)^{\times}$ the reduction of $\chi_p$ modulo $p$. We denote by $(\Z/p\Z)(1)$ the $G_{\Q}$-module consisting of $\Z/p\Z$ with the action of $\overline{\chi}_p$. We denote by $\mu_p$ the $G_{\Q}$-module consisting of $p$th roots of unity in $\overline{\Q}$. Note that we have a non-canonical isomorphism of $G_{\Q}$-modules $(\Z/p\Z)(1) \simeq \mu_p$.

We denote by $g_p\geq 1$ the $\Z_p$-rank of $\mathbb{T}_{\mathfrak{P}}^0$ (the completion of the cuspidal Hecke algebra at the $p$-Eisenstein prime $\mathfrak{P}$). Recall that $\mathbb{T}_{\mathfrak{P}}^0$ is a DVR since we have assumed $p \mid \mid N-1$. We fix a local generator $\eta$ of $I$, \ie $\eta \in I$ is such that $I \cdot  \mathbb{T}_{\mathfrak{P}}^0 = \eta \cdot  \mathbb{T}_{\mathfrak{P}}^0 $.  

We assume as before that $N$ splits in $K$, and we write $N\mathcal{O}_K = \mathfrak{N}_1 \mathfrak{N}_2$ for prime ideals $\mathfrak{N}_1$, $\mathfrak{N}_2$ of $\mathcal{O}_K$. Recall that $u(K)$ is a fundamental unit of $\mathcal{O}_K$ and $h(K)$ is its class number.

\subsection{Background on the Galois structure of $\tilde{J}^{(p)}[I^2]$}\label{subsec_galois}
In this paragraph, we state some facts about the $G_{\Q}$-module $\tilde{J}^{(p)}[I^2] \otimes \Z_p$. These are well-known to experts on Mazur's Eisenstein ideal (\cf \eg \cite{Calegari_Emerton}, \cite{WWE}). Mazur proved \cite[Corollary II.16.4]{Mazur_Eisenstein} that 
\begin{equation}\label{eq_Mazur_Shimura_cuspidal}
\tilde{J}^{(p)}[I] \otimes \Z_p = \Sigma_p \oplus C_p\text{ ,}
\end{equation}
 where $C_p \simeq \Z/p\Z$ is the $p$-part of the cuspidal subgroup and $\Sigma_p$ is the $p$-part of the Shimura subgroup. 

 Let us recall that $\Sigma_p$ corresponds to $p$-part of the image in $\tilde{J}^{(p)}$ of the kernel of the natural map $J_0(N) \rightarrow J_1(N)$. By \cite[Proposition II.11.6]{Mazur_Eisenstein}, we have a canonical isomorphism of $G_{\Q}$-modules
 \begin{equation}\label{eq_Shimura_subgp}
 \Sigma_p \simeq \Hom((\Z/N\Z)^{\times}, \mu_p) \text{ .}
 \end{equation}

 Choose a $\mathbb{T}^0_{\mathfrak{P}}/I^2$-basis $(e_1, e_2)$ of $\tilde{J}^{(p)}[I^2] \otimes \Z_p$ such that $I\cdot e_1 = \Sigma_p$ and $I\cdot e_2 = C_p$. We choose $e_2$ so that $\eta\cdot e_2$ is the canonical generator of $C_p$ (corresponding to $1 \in \Z/p\Z$).
 This choice of a basis yields a Galois representation
$$\rho : G_{\Q} \rightarrow \GL_2(\mathbb{T}^0_{\mathfrak{P}}/I^2)$$
given by $\rho = \begin{pmatrix} a & \eta b \\ \eta c & d \end{pmatrix}$. Note that $\det(\rho) = \chi_p$ and $\rho$ is unramified outisde $N$ and $p$.

We see that $a$ and $d$ are characters $G_{\Q} \rightarrow (\mathbb{T}^0_{\mathfrak{P}}/I^2)^{\times}$ and $a\cdot d = \det(\rho) = \chi_p$. Write $a = \chi_p\cdot (1+\eta \cdot \varphi)$ and $d = 1-\eta \cdot \varphi$ where $\varphi : G_{\Q} \rightarrow \mathbb{T}^0_{\mathfrak{P}}/I = \Z/p\Z$ is a group homomorphism.

Mazur proved \cite[Proposition II.18.9]{Mazur_Eisenstein} that there is a canonical group isomorphism
$$
I/I^2 \otimes \Z_p \xrightarrow{\sim} (\Z/N\Z)^{\times} \otimes \Z_p
$$
given by
$$T_{\ell}-\ell-1 \mapsto \ell^{\ell-1} \otimes 1$$
for all primes $\ell \neq N$ (we are using the fact that $p>2$; otherwise one needs to use $\ell^{\frac{\ell-1}{2}}$ if $\ell \neq 2$). 

The choice of our local generator $\eta \in I$ determines a generator of $(\Z/N\Z)^{\times} \otimes \Z_p$, \ie a surjective group homomorphism $\log : (\Z/N\Z)^{\times} \rightarrow \Z/p\Z$. By (\ref{eq_Shimura_subgp}), the choice of $\eta$ induces a \emph{canonical} isomorphism
\begin{equation}\label{eq_Shimura_subgp_2}
\Sigma_p \xrightarrow{\sim} \mu_p \text{ .}
\end{equation}

Let us note that the choice of $e_1$ determines a choice of a generator of $\mu_p$ via (\ref{eq_Shimura_subgp_2}), namely $\eta\cdot e_1$. We denote this generator of $\mu_p$ by 
\begin{equation}\label{eq_def_zeta_p}
\zeta_p \in \overline{\Q} \text{ .}
\end{equation}

Mazur's isomorphism may be rewritten as 
$$I/I^2 \otimes \Z_p \xrightarrow{\sim} \Z/p\Z$$
$$T_{\ell}-\ell-1 \mapsto (\ell-1)\log(\ell) $$
$$\eta \mapsto 1 \text{ .}$$

On may view $\log$ as a group homomorphism $G_{\Q} \rightarrow \Z/p\Z$ via the identification $\Gal(\Q(\zeta_N)/\Q) \simeq (\Z/N\Z)^{\times}$.  If $\ell \neq N$ is a prime and $\Frob_{\ell} \in G_{\Q}$ is a Frobenius at $\ell$, we have (by the Eichler--Shimura relation)

$$T_{\ell} = \Tr(\rho(\Frob_{\ell})) = \ell+1 + \eta (\ell-1)\varphi(\Frob_{\ell})  \text{ .}$$
Therefore, we have $$\varphi(\Frob_{\ell}) =\log(\ell) \text{ .}$$

We shall need the following consequence of the description of $\rho$ given above. 
\begin{prop}\label{proof_torsion_I^2}
Let $F$ be a number field such that $F \cap \Q(\zeta_{Np}) = \Q$. Then $\tilde{J}^{(p)}(F)[I^2] = C_p$.
\end{prop}
\begin{proof}
Let $Q \in \tilde{J}^{(p)}(F)[I^2]$. Then $\eta \cdot Q \in \tilde{J}^{(p)}(F)[I]$ so $\eta\cdot Q$ belongs to $\tilde{J}^{(p)}(F) \cap (C_p \oplus \Sigma_p)$. Since $\Sigma_p \simeq \mu_p$ and $F \cap \Q(\zeta_p) = \Q$, we conclude that $\eta \cdot Q \in C_p$. One may assume that $Q = e_2$ (our basis element of $\tilde{J}^{(p)}[I^2]$). But for $g \in G_{\Q}$, we have $g(e_2)-e_2 = \eta\cdot b(g)\cdot e_1 - \eta\cdot \log(g)\cdot e_2$. Since $F \cap \Q(\zeta_N) = \Q$, there exists $g \in G_F$ such that $\log(g) \neq 0$. Therefore, $g(e_2) \neq e_2$ and $e_2$ cannot be defined over $F$.
\end{proof}

\subsection{Reduction to $\Sel_{I}(\tilde{J}^{(p)}/K)$}
In this paragraph, we make some reductions to replace $\Sel_{p}(\tilde{J}_D^{(p)}/\Q)_{\mathfrak{P}}$ with the ``simpler'' Selmer group $\Sel_{I}(\tilde{J}^{(p)}/K)$ (defined below). We first relate $\Sel_{p}(\tilde{J}_D^{(p)}/\Q)$ and $\Sel_{p}(\tilde{J}^{(p)}/K)$. 

\begin{lem}\label{lemma_K_Q_Selmer}
We have a natural group isomorphism
$$\Sel_{p}(\tilde{J}^{(p)}/K) \simeq \Sel_{p}(\tilde{J}^{(p)}/\Q) \oplus \Sel_{p}(\tilde{J}_D^{(p)}/\Q) \text{ .}$$
\end{lem}
\begin{proof}
We shall use the following general fact. Let $A$ be a commutative ring on which $2$ is invertible, $G$ be a group and $H$ be a normal subgroup of $G$ with $G/H \simeq \Z/2\Z$. If $M$ is a $A$-module with an action of $G$, then there is a natural $G$-equivariant isomorphism of $A$-modules
$$\Ind_H^G M \simeq M \oplus (M \otimes \chi)\text{ ,}$$
where $\chi : G/H \rightarrow \{1, -1\}$ is the character attached to $H$. This follows easily from the definition of the induced module.

Since $p$ is odd, we get an isomorphism of $G_{\Q}$-modules
\begin{equation}\label{eq_induced_1}
\Ind_{G_K}^{G_{\Q}} \tilde{J}^{(p)}[p] \simeq \tilde{J}^{(p)}[p] \oplus \tilde{J}_D^{(p)}[p] \text{ .}
\end{equation}
Similarly, for any place $v$ of $\Q$, we have an isomorphism of $G_{\Q_v}$-modules
\begin{equation}\label{eq_induced_2}
\bigoplus_{w \mid v \atop w \text{ place of } K} \Ind_{G_{K_w}}^{G_{\Q_v}} \tilde{J}^{(p)}(\overline{\Q}_w) \otimes \Z_p \simeq \tilde{J}^{(p)}(\overline{\Q}_v) \otimes \Z_p \bigoplus \tilde{J}_D^{(p)}(\overline{\Q}_v) \otimes \Z_p 
\end{equation}
(in the case where $v$ is inert of ramified, we use the general fact stated above.)
Combining (\ref{eq_induced_1}) and (\ref{eq_induced_2}), we get an isomorphism of $G_{\Q_v}$-modules
\begin{equation}\label{eq_induced_3}
\left(\Ind_{G_K}^{G_{\Q}} \tilde{J}^{(p)}[p]\right) \mid_{G_{\Q_v}}\simeq\bigoplus_{w \mid v \atop w \text{ place of } K} \Ind_{G_{K_w}}^{G_{\Q_v}} \tilde{J}^{(p)}[p] \text{ .}
\end{equation}
This is a special case of Mackey's formula for the restriction of an induced.
The following diagram is commutative:
\[
\begin{tikzcd}[trim left,trim right=0cm]
	{H^1(G_K, \tilde{J}^{(p)}[p])} & {} & {\bigoplus_{w \mid v} H^1(G_{K_w}, \tilde{J}^{(p)}(\overline{\Q}_w) \otimes \Z_p)} \\
	{H^1(G_{\Q}, \Ind_{G_K}^{G_{\Q}}\tilde{J}^{(p)}[p])} && {\bigoplus_{w \mid v} H^1(G_{\Q_v}, \Ind_{G_{K_w}}^{G_{\Q_v} }\tilde{J}^{(p)}(\overline{K}_w) \otimes \Z_p)} \\
	{H^1(G_{\Q}, \tilde{J}^{(p)}[p]) \bigoplus H^1(G_{\Q}, \tilde{J}_D^{(p)}[p])} && {H^1(G_{\Q_v}, \tilde{J}^{(p)}(\overline{\Q}_v) \otimes \Z_p) \bigoplus H^1(G_{\Q_v}, \tilde{J}_D^{(p)}(\overline{\Q}_v) \otimes \Z_p).}
	\arrow["{\bigoplus_{w \mid v} \text{Res}_w}", from=1-1, to=1-3]
	\arrow["{\text{Shapiro}}"', from=1-1, to=2-1]
	\arrow["(\ref{eq_induced_1})"', from=2-1, to=3-1]
	\arrow["{\text{Res}_v}", from=3-1, to=3-3]
	\arrow["(\ref{eq_induced_2})", from=2-3, to=3-3]
	\arrow["{\text{Shapiro}}", from=1-3, to=2-3]
	\arrow["(\ref{eq_induced_3})", from=2-1, to=2-3]
\end{tikzcd}
\]
This proves that $$\Sel_{p}(\tilde{J}^{(p)}/K) \simeq \Sel_{p}(\tilde{J}^{(p)}/\Q) \oplus \Sel_{p}(\tilde{J}_D^{(p)}/\Q) \text{ .}$$
\end{proof}

Let us now define, for every integer $n\geq 1$, another Selmer group, denoted by $\Sel_{I^n}(\tilde{J}^{(p)}/K)$, which will be easier to understand (we shall mainly be interested in the case $n=1$, but the definition will be useful in the proof of Lemma \ref{lem_passage_I_p}). Recall that we have fixed $\eta \in I$ such that $\eta$ generates $I \cdot \mathbb{T}_{\mathfrak{P}}^0$. We define $$\Sel_{I^n}(\tilde{J}^{(p)}/K) \subset H^1(G_K, \tilde{J}^{(p)}[I^n] \otimes \Z_p)$$
by imposing that locally at every place $w$ of $K$, the restriction of $\Sel_{I^n}(\tilde{J}^{(p)}/K)$ at $w$ lies in the image of the Kummer map
$$\kappa_w : \left(\tilde{J}^{(p)}(K_w)/I^n\tilde{J}^{(p)}(K_w) \right) \otimes \Z_p \rightarrow H^1(G_{K_w}, \tilde{J}^{(p)}[I^n] \otimes \Z_p) $$
coming from the short exact sequence 
$$0 \rightarrow \tilde{J}^{(p)}[\eta^n] \rightarrow \tilde{J}^{(p)} \xrightarrow{\eta^n} \tilde{J}^{(p)} \rightarrow 0 \text{ .}$$
Here, we use the fact that $ \tilde{J}^{(p)}[I^n] \otimes \Z_p = ( \tilde{J}^{(p)}[\eta^n])_{\mathfrak{P}}$ and $$\left( \tilde{J}^{(p)}(K_w)/I^n\tilde{J}^{(p)}(K_w) \right) \otimes \Z_p = \left( \tilde{J}^{(p)}(K_w)/\eta^n\tilde{J}^{(p)}(K_w) \right)_{\mathfrak{P}} \text{ .}$$
Similarly, we have an embedding
\begin{equation}\label{eq_global_MW_I}
\left( \tilde{J}^{(p)}(K)/I^n \tilde{J}^{(p)}(K) \right) \otimes \Z_p = \left( \tilde{J}^{(p)}(K)/\eta^n \tilde{J}^{(p)}(K) \right)_{\mathfrak{P}} \hookrightarrow \Sel_{I^n}(\tilde{J}^{(p)}/K) \text{ .}
\end{equation}

The definition of $\Sel_{I^n}(\tilde{J}^{(p)}/K)$ does not depend on the choice of $\eta$. Note that if $\mathbb{T}_{\mathfrak{P}}^0 = \Z_p$ (\ie $g_p= \rk_{\Z_p} \mathbb{T}_{\mathfrak{P}}^0 = 1$) then one can choose $\eta = p$ and we get $\Sel_I(\tilde{J}^{(p)}/K) = \Sel_p(\tilde{J}^{(p)}/K)_{\mathfrak{P}}$. More generally, we have $\Sel_p(\tilde{J}^{(p)}/K)_{\mathfrak{P}} = \Sel_{I^{g_p}}(\tilde{J}^{(p)}/K)$.

\begin{lem}\label{lem_passage_I_p}
The cuspidal subgroup $C_p \subset \tilde{J}^{(p)}(\Q)[I]\subset \tilde{J}^{(p)}(K)[I] $ yields a non-zero subgroup of $\Sel_I(\tilde{J}^{(p)}/K)$. Furthermore, we have $\Sel_{p}(\tilde{J}_D^{(p)}/\Q)_{\mathfrak{P}} \neq 0$ if and only if $\rk_{\mathbf{F}_p} \Sel_I(\tilde{J}^{(p)}/K)>1$.
\end{lem}
\begin{proof}
For all $n\geq 1$, the cuspidal subgroup $C_p$ yields a subgroup of $ \left( \tilde{J}^{(p)}(K)/I^n \tilde{J}^{(p)}(K) \right) \otimes \Z_p$, which we denote by $\overline{C}_p^{(n)}$. We claim that $\overline{C}_p^{(n)}$ is non-zero, and thus yields a non-zero element of $\Sel_{I^n}(\tilde{J}^{(p)}/K)$ by (\ref{eq_global_MW_I}). Indeed, if $\overline{C}_p^{(n)} = 0$ then $C_p \subset I^n \cdot \tilde{J}^{(p)}(K)$. This implies that $ \tilde{J}^{(p)}(K)[I^2]$ contains strictly $C_p$, which contradicts Proposition \ref{proof_torsion_I^2} applied to $F=K$ (we use the fact that $p$ and $N$ are unramified in $K$). We have thus proved that $\overline{C}_p^{(n)}$ is a non-zero subgroup of $\Sel_{I^n}(\tilde{J}^{(p)}/K)$. 

Let us prove the second part of the lemma. 
By the work of Mazur \cite{Mazur_1979,Mazur_Eisenstein}, we know that 
\begin{itemize}
\item The Mordell--Weil rank of $\tilde{J}^{(p)}(\Q)$ is zero.
\item The completion at $\mathfrak{P}$ of the Tate--Shafarevich group of $\tilde{J}^{(p)}$ over $\Q$ is trivial.
\item The $p$-part of the torsion in $\tilde{J}^{(p)}(\Q)$ is equal to the cuspidal subgroup $C_p$ (which is cyclic of order $p$ since we assume $p \mid \mid N-1$).
\end{itemize}
Therefore, we conclude that $\Sel_{p}(\tilde{J}^{(p)}/\Q)_{\mathfrak{P}} \simeq \Z/p\Z$. By Lemma \ref{lemma_K_Q_Selmer}, we get that $\Sel_{p}(\tilde{J}_D^{(p)}/\Q)_{\mathfrak{P}} \neq 0$ if and only if $\rk_{\mathbf{F}_p} \Sel_p(\tilde{J}^{(p)}/K)_{\mathfrak{P}}>1$. If $g_p=1$, we have $\Sel_p(\tilde{J}^{(p)}/K)_{\mathfrak{P}} =\Sel_I(\tilde{J}^{(p)}/K)$ so we are done. Let us therefore assume that $g_p>1$ in what follows.

Consider the short exact sequence
$$ 0 \rightarrow \tilde{J}^{(p)}[\eta] \rightarrow \tilde{J}^{(p)}[\eta^{n+1}] \xrightarrow{\eta}  \tilde{J}^{(p)}[\eta^n] \rightarrow 0 \text{ .}$$
One easily checks that completing at $\mathfrak{P}$ and taking Galois cohomology, we get an exact sequence
$$ \Sel_I(\tilde{J}^{(p)}/K) \rightarrow \Sel_{I^{n+1}}(\tilde{J}^{(p)}/K) \rightarrow \Sel_{I^n}(\tilde{J}^{(p)}/K) \text{ .}$$

We claim that the kernel of $ \Sel_I(\tilde{J}^{(p)}/K) \rightarrow \Sel_{I^{n+1}}(\tilde{J}^{(p)}/K) $ is generated by $\overline{C}_p^{(1)}$. Indeed, let $\alpha : G_K \rightarrow \tilde{J}^{(p)}[I] \otimes \Z_p$ be a cocycle whose class lies in the kernel. This means that, up to scaling $\alpha$, there exists $P \in \tilde{J}^{(p)}[I^{n+1}]$ such that for all $g \in G_K$, we have $\alpha(g) = gP-P$. Since $\eta \cdot \alpha(g)=0$, we get $Q := \eta \cdot P \in  \tilde{J}^{(p)}(K)$, and therefore $Q \in  \tilde{J}^{(p)}(K)[I^{n+1}]$. By Proposition \ref{proof_torsion_I^2}, we have $ \tilde{J}^{(p)}(K)[I^{n+1}] = C_p$. So $\alpha$ belongs to $\overline{C}_p^{(1)}$.

For all $n\geq 1$, we thus have an exact sequence 
$$0 \rightarrow \overline{C}_p^{(1)} \rightarrow \Sel_I(\tilde{J}^{(p)}/K) \rightarrow \Sel_{I^{n+1}}(\tilde{J}^{(p)}/K) \rightarrow \Sel_{I^n}(\tilde{J}^{(p)}/K) \text{ .}$$
Recall that our goal is to prove that $\rk_{\mathbf{F}_p}\Sel_I(\tilde{J}^{(p)}/K)  >1$ if and only if $\rk_{\mathbf{F}_p}\Sel_p(\tilde{J}^{(p)}/K)_{\mathfrak{P}}>1$. More generally, let us prove that for all $n\geq 1$,  $\rk_{\mathbf{F}_p}\Sel_I(\tilde{J}^{(p)}/K)  >1$ if and only if $\rk_{\mathbf{F}_p}\Sel_{I^n}(\tilde{J}^{(p)}/K) >1$. We prove this by induction on $n$. In view of the above exact sequence, to prove the induction step it suffices to show that an element of $ \Sel_I(\tilde{J}^{(p)}/K)$ not in $\overline{C}_p^{(1)}$ is mapped to an element of $\Sel_{I^{n+1}}(\tilde{J}^{(p)}/K)$ not in $\overline{C}_p^{(n+1)}$.

Let $P \in \tilde{J}^{(p)}(\overline{K})[I^{n+1}]$ such that $\eta^{n+1} \cdot P$ is a generator of $C_p$. A generator of $\overline{C}_p^{(n+1)}$ is given by the cocycle $\beta : G_K \rightarrow \tilde{J}^{(p)}[I^{n+1}] \otimes \Z_p$ defined by $g \mapsto g(P)-P$. Let $\alpha : G_K \rightarrow \tilde{J}^{(p)}[I] \otimes \Z_p$ be a cocycle whose class in $\Sel_I(\tilde{J}^{(p)}/K)$ is mapped to the class of $\beta$ in $\Sel_{I^{n+1}}(\tilde{J}^{(p)}/K)$. This implies that there exists $Q \in  \tilde{J}^{(p)}[I^{n+1}] \otimes \Z_p$ such that for all $g \in G_K$, we have $$\alpha(g) = g(P)-P - (g(Q)-Q) \in  \tilde{J}^{(p)}[I] \otimes \Z_p \text{ .}$$
Up to scaling $Q$ by an integer prime to $p$, one can assume that $Q \in  \tilde{J}^{(p)}[I^{n+1}]$. Furthermore, we may assume (up to replacing $P$ with $P-Q$) that for all $g \in G_K$, we have $g(\eta\cdot P) = \eta\cdot P$. This means $\eta \cdot P \in \tilde{J}^{(p)}(K)$. Since $\eta^{n+1} \cdot P$ is a generator of $C_p$, we get that $C_p \subset \eta^{n}\cdot \tilde{J}^{(p)}(K)$ and thus $C_p \subset I\cdot \tilde{J}^{(p)}(K)$. Consequently, we have $\tilde{J}^{(p)}(K)[I^2] \neq C_p$. This contradicts Proposition \ref{proof_torsion_I^2}.

\end{proof}

\subsection{Computation of  $\Sel_{I}(\tilde{J}^{(p)}/K)$}
Our goal in this final section is to prove the following:
\begin{equation}\label{eq_main_thm_selmer}
\begin{aligned}
& \rk_{\mathbf{F}_p} \Sel_I(\tilde{J}^{(p)}/K)>1 \\&  \Longleftrightarrow \\& u(K)^{h(K)} \text{ is a } p\text{th power modulo } \mathfrak{N}_i \text{ (for any, equivalently all, } i \in \{1,2\}\text{).} 
\end{aligned}
\end{equation}
By Lemma \ref{lem_passage_I_p}, this will prove the equivalence of (\ref{main_thm_bis_2}) and (\ref{main_thm_bis_3}) in Theorem \ref{main_thm_bis}. 

Recall that $\Sel_I(\tilde{J}^{(p)}/K)$ is a subgroup of $H^1(G_{K, Np}, \tilde{J}^{(p)}[I] \otimes \Z_p)$ defined by some local conditions at $\mathfrak{N}_1$, $\mathfrak{N}_2$ and the primes above $p$ in $\mathcal{O}_K$ (at other places the condition is unramified, whence the use of the Galois group $G_{K, Np}$ of the maximal unramified outside $Np$ extension of $K$). By our choice of $\eta \in I$, we have by  (\ref{eq_Mazur_Shimura_cuspidal}) and (\ref{eq_Shimura_subgp_2})
$$\tilde{J}^{(p)}[I] \otimes \Z_p  = \Sigma_p \oplus C_p \simeq \mu_p \oplus \Z/p\Z \text{ .}$$

One thus gets a natural inclusion $$\Sel_I(\tilde{J}^{(p)}/K) \subset H^1(G_{K,Np}, \mu_p) \bigoplus H^1(G_{K,Np}, \Z/p\Z) \text{ .}$$

The local condition at $p$ has been determined by \v{C}esnavi\v{c}ius in \cite{Ke}. 

\begin{lem}
The Selmer condition for $\Sel_I(\tilde{J}^{(p)}/K)$ at a place $w$ of $K$ dividing $p$ is the fppf condition, \ie the subgroup 
$$ H^1_{\fppf}(\mathcal{O}_w, \mu_p) \bigoplus H^1_{\etale}(\mathcal{O}_w, \Z/p\Z) \subset  H^1(G_{K_w}, \mu_p) \bigoplus H^1(G_{K_w}, \Z/p\Z) \text{ .}$$
\end{lem}
\begin{proof}
We apply \cite[Proposition 2.5 (d)]{Ke} to $A=B=\tilde{J}^{(p)}$ over $k=K_w$, with the isogeny $\phi = \eta$ (a local generator of the Eisenstein ideal $I$ as before). Since $\tilde{J}^{(p)}$ has good reduction at $w$, its N\'eron model $\mathcal{J}$ over $\mathcal{O}_w$ (the valuation ring of $K_w$) is a connected proper abelian scheme over $\mathcal{O}_w$. Therefore, the Tamagawa numbers $c_A$ and $c_B$ are equal to one. The abelian scheme $\mathcal{J}[\eta]$ is a direct sum of its prime to $p$-part and its $p$-part. The $p$-part of $\mathcal{J}[\eta]$ decomposes according to the maximal ideals of residual characteristic $p$ of $\mathbb{T}^0$ (indeed, there is an idempotent of $\mathbb{T}^0 \otimes_{\Z} \Z_p$ associated to each such maximal ideal). The $\mathfrak{P}$-part of that idempotent decomposition is $\Z/p\Z \oplus \mu_p$. The lemma then follows directly from \cite[Proposition 2.5 (d)]{Ke} (after taking the $p$-part and taking the decomposition according to the idempotents just mentioned).
\end{proof}

Therefore, $\Sel_I(\tilde{J}^{(p)}/K)$ is a subgroup of 
$$ H^1(\mathcal{O}_K[1/N], \mu_p) \bigoplus H^1(G_{K,N}, \Z/p\Z)\text{ ,}$$
where $H^1(\mathcal{O}_K[1/N], \mu_p)$ is the subgroup of $H^1(G_{K,Np}, \mu_p)$ which locally at $w \mid p$ lies in $$H^1_{\fppf}(\mathcal{O}_w, \mu_p) \simeq \mathcal{O}_w^{\times}/(\mathcal{O}_w^{\times})^p$$
(via Kummer theory). 

By excision and Kummer theory, one can see that
$$H^1(\mathcal{O}_K[1/N], \mu_p)\cong H^1_{\fppf}(\mathcal{O}_K[1/N], \mu_p) \text{ ,}$$
where  $H^1_{\fppf}$ denotes the flat cohomology (\cf \cite[\href{https://stacks.math.columbia.edu/tag/03PK}{Tag 03PK}]{stacks-project} for basic definitions and facts about flat cohomology, and \cf \cite[Theorem 3.6]{Schmidt} for the above identification). Kummer theory for flat cohomology gives an exact sequence

\begin{equation}\label{eq_Kummer_fppf}
0\rightarrow \mathcal{O}_K[1/N]^{\times}/(\mathcal{O}_K[1/N]^{\times})^p \rightarrow H^1(\mathcal{O}_K[1/N], \mu_p) \rightarrow \Pic(\mathcal{O}_K[1/N])[p] \rightarrow 0 \text{ .}
\end{equation}

We now determine precisely the local condition at $\mathfrak{N}_1$ and $\mathfrak{N}_2$ cutting out $\Sel_I(\tilde{J}^{(p)}/K)$ in $$H^1(\mathcal{O}_K[1/N], \mu_p) \bigoplus H^1(G_{K,N}, \Z/p\Z)  \text{ .}$$
That is, for $i \in \{1, 2\}$ we describe the image of $(\tilde{J}^{(p)}(K_{\mathfrak{N}_i})/I\tilde{J}^{(p)}(K_{\mathfrak{N}_i})) \otimes \Z_p$ by the Kummer map $\kappa_{\mathfrak{N}_i}$ in $$H^1(G_{K_{\mathfrak{N}_i}}, \tilde{J}^{(p)}[I] \otimes \Z_p) =H^1(G_{K_{\mathfrak{N}_i}}, \mu_p) \bigoplus  H^1(G_{K_{\mathfrak{N}_i}}, \Z/p\Z) \text{ .}$$

Recall that in \S \ref{subsec_galois}, we have chosen a basis $(e_1,e_2)$ of $\tilde{J}^{(p)}[I^2] \otimes \Z_p$ such that the corresponding Galois representation $\rho : G_{\Q} \rightarrow \GL_2(\mathbb{T}_{\mathfrak{P}}^0/I^2)$ is of the form
\begin{equation}\label{eq_description_rho}
\rho = \begin{pmatrix} \chi_p(1+\eta \log) & \eta b \\ \eta c & 1-\eta\log \end{pmatrix} \text{ .}
\end{equation}
Here, $b$, $c$ and $\log$ are maps $G_{\Q} \rightarrow \Z/p\Z$. More precisely, $b$ is a cocycle in $Z^1(G_{\Q}, (\Z/p\Z)(1))$, $c$ is a cocycle in $Z^1(G_{\Q}, (\Z/p\Z)(-1))$ and $\log$ is a cocycle in $Z^1(G_{\Q}, \Z/p\Z)$.

For $i \in \{1,2\}$, we have fixed a decomposition group $G_{K_{\mathfrak{N}_i}}$, \ie an embedding $\overline{K} \hookrightarrow \overline{K}_{\mathfrak{N}_i}$. Therefore, one can view $\mu_p \subset \overline{\Q}^{\times}$ as a subgroup of $\overline{K}_{\mathfrak{N}_i}^{\times}$. Since we have made a choice of $\log : (\Z/N\Z)^{\times} \rightarrow \Z/p\Z$ (induced by the choice of $\eta \in I$), we have a canonical generator of $\mu_p((\Z/N\Z)^{\times})$, whose $\log$ is equal to $\frac{N-1}{p}$. Its Teichm\"uller lift yields a a canonical generator of $\mu_p(\overline{K}_{\mathfrak{N}_i})$. Thus, for $i\in \{1,2\}$ we have a natural isomorphism of $G_{K_{\mathfrak{N}_i}}$-modules (depending on the choice of $\eta$)
\begin{equation}\label{eq_local_N_mu_iso}
\mu_p \xrightarrow{\sim} \Z/p\Z \text{ .}
\end{equation}

We shall need the following result in our computations.
\begin{lem}\label{lemma_log_N_comparison}
Under the Kummer isomorphism $H^1(G_{K_{\mathfrak{N}_i}}, \mu_p) \simeq K_{\mathfrak{N}_i}^{\times} \otimes \Z/p\Z$ and the identification (\ref{eq_local_N_mu_iso}), the class of $N \otimes 1$ corresponds to the class of the cocycle $- \log$ in $H^1(G_{K_{\mathfrak{N}_i}}, \Z/p\Z)$.
\end{lem}
\begin{proof}
Let $F=K_{\mathfrak{N}_i} = \Q_N$ and  $\text{Art} : F^{\times} \rightarrow \Gal(F^{\text{ab}}/F)$ be the Artin map (sending a uniformizer to an arithmetic Frobenius element). The cocycle (in this case group homomorphism) associated to $N$ corresponds to the map $F^{\times} \rightarrow \mu_p(F)$ given by $a \mapsto \frac{\text{Art}(a)(N^{1/p})}{N^{1/p}}$. The explicit reciprocity law of local class field theory tells us that $\frac{\text{Art}(a)(N^{1/p})}{N^{1/p}} \equiv u^{-\frac{N-1}{p}} \text{ (modulo } N\text{)}$ where $a = u\cdot N^v$ for some $u \in \Z_N^{\times}$ and $v \in \Z_{\geq 0}$. Thus, under the identification (\ref{eq_local_N_mu_iso}) given by $\log$, the Kummer class of $N$ corresponds to a group homomorphism $G_F \rightarrow \Z/p\Z$ given by $\text{Art}(a) \mapsto -\log(u)$, which concludes the proof of the lemma.
\end{proof}

Recall that we had in (\ref{eq_def_zeta_p}), based on our choices of $(e_1, e_2)$ and $\eta \in I$, a natural generator $\zeta_p \in \mu_p$. For $i \in \{1, 2\}$, we denote by 
\begin{equation}\label{eq_coeff_zeta_p_local}
a_i \in (\Z/p\Z)^{\times}
\end{equation}
the image of $\zeta_p$ via (\ref{eq_local_N_mu_iso}).

Using (\ref{eq_local_N_mu_iso}), one can view $H^1(G_{K_{\mathfrak{N}_i}}, \tilde{J}^{(p)}[I] \otimes \Z_p)$ as a subgroup of 
$$H^1(G_{K_{\mathfrak{N}_i}}, \Z/p\Z) \bigoplus H^1(G_{K_{\mathfrak{N}_i}}, \Z/p\Z) \text{ ,}$$
where the first copy of $H^1(G_{K_{\mathfrak{N}_i}}, \Z/p\Z)$ corresponding to $\Sigma_p$ and the second copy to $C_p$.

By Proposition \ref{thm_uniformization}, the group $(\tilde{J}^{(p)}(K_{\mathfrak{N}_i})/I\tilde{J}^{(p)}(K_{\mathfrak{N}_i})) \otimes \Z_p$ is isomorphic to $(\Z/p\Z)^2$, generated by $\Sigma_p$ and $C_p$, \ie by $\eta \cdot e_1$ and $\eta \cdot e_2$.  Let $\kappa_{\Sigma}$ and $\kappa_{C}$ be the classes of cocycles in $H^1(G_{K_{\mathfrak{N}_i}}, \tilde{J}^{(p)}[I] \otimes \Z_p)$ corresponding to $\eta \cdot e_1$ and $\eta \cdot e_2$ respectively (these depends implicitly on $i \in \{1,2\}$). 

By definition of the Kummer map, for all $g \in G_{K_{\mathfrak{N}_i}}$ we have $$\kappa_{\Sigma}(g) = g(e_1)-e_1$$ and 
\begin{equation}\label{eq_kappa_C}
\kappa_{C}(g) = g(e_2)-e_2
\end{equation}
(as cocycles). The image of the Kummer map $\kappa_{\mathfrak{N}_i}$ in $$H^1(G_{K_{\mathfrak{N}_i}}, \Z/p\Z) \bigoplus H^1(G_{K_{\mathfrak{N}_i}}, \Z/p\Z)$$ is the set of classes of cocycles of the form
$$\lambda_i \cdot \kappa_{C} + \mu_i \cdot \kappa_{\Sigma}$$
for $\lambda_i, \mu_i \in \Z/p\Z$. The description of these cocycles actually depend on whether $g_p=1$ or $g_p>1$ (where we recall that $g_p = \rk_{\Z_p} \mathbb{T}_{\mathfrak{P}}^0$). These two cases will require similar but different computations. We thus consider them separately.

\subsubsection{Case $g_p=1$}

Assume in this paragraph that $g_p=1$. In this case, one may choose $\eta=p$ and we have $\mathbb{T}_{\mathfrak{P}}^0/I^2 = \Z/p^2\Z$. 
The cyclotomic character (modulo $I^2$) $\chi_p : G_{K_{\mathfrak{N}_i}} \rightarrow (\Z/p^2\Z)^{\times}$ can be written as $$\chi_p = 1 + \eta\cdot \phi$$ where $\phi : G_{K_{\mathfrak{N}_i}} \rightarrow \Z/p\Z$ is an unramified group homomorphism. The morphism $\phi$ (which strictly speaking depends on $i \in \{1,2\}$) is characterized by 
$$\phi(\Frob_{\mathfrak{N}_i}) = \frac{N-1}{p} \text{ ,}$$
where $\Frob_{\mathfrak{N}_i}$ is an arithmetic Frobenius element at $\mathfrak{N}_i$.

We shall need the following local description of the cocycles $b$ and $c$ involved in (\ref{eq_description_rho}).
\begin{lem}\label{lemma_description_b_c_g_p=1}
There exists $k\in (\Z/p\Z)^{\times}$ such that for any $i \in \{1,2\}$, the restriction to $G_{K_{\mathfrak{N}_i}}$ of $b$ and $c$ satisfy
$$b = ka_i^{-1}\cdot \log$$
and
$$c = -k^{-1}a_i\cdot \log - k^{-1}a_i \cdot \phi \text{ .}$$
\end{lem}
\begin{proof}
The class of the cocycle $b \in Z^1(G_{\Q}, (\Z/p\Z)(1))$ yields a class in $H^1(G_{\Q, Np}, \mu_p)$ since we have fixed a global $p$th root of unity $\zeta_p \in \mu_p$. Since the representation $\rho$ is finite flat at $p$, we see that the class of $b$ is in $H^1_{\fppf}(\Z[1/N], \mu_p)$, \ie in $\Z[1/N]^{\times} \otimes \Z/p\Z$. If $b$ is trivial in $H^1_{\fppf}(\Z[1/N], \mu_p)$, then we may assume that $e_2$ if fixed by $G_{\Q}$, which implies that $\tilde{J}^{(p)}(\Q)[p^2]$ contains a subgroup isomorphic to $\Z/p^2\Z$. This is a contradiction since Mazur proved that $\tilde{J}^{(p)}(\Q)[p^2] \simeq \Z/p\Z$. 

Thus, $b$ represents a non-trivial class in $H^1_{\fppf}(\Z[1/N], \mu_p)$, which must correspond to the Kummer class of $N \otimes k \in \Z[1/N]^{\times} \otimes \Z/p\Z$ for some $k \in (\Z/p\Z)^{\times}$. By Lemma \ref{lemma_log_N_comparison}, we have $a_i \cdot b = -k \cdot \log$ locally at $\mathfrak{N}_i$. Up to replacing $k$ with $-k$, we get $b = ka_i^{-1} \cdot \log$, as wanted.

Let us now prove that $c = -k^{-1}a_i\cdot \log - k^{-1}a_i \cdot \phi$ locally at $\mathfrak{N}_i$. Recall that $\rho : G_{\Q} \rightarrow \GL_2(\mathbb{T}_{\mathfrak{P}}^0/I^2)$ is given by $\rho = \begin{pmatrix} \chi_p(1+\eta\cdot \log) & \eta b \\ \eta c & 1-\eta \cdot \log \end{pmatrix}$. 

Since $J_0(N)$ has semi-stable reduction at $N$, we know that there is a quotient line in $\tilde{J}^{(p)}[I^2]$ fixed by $G_{K_{\mathfrak{N}_i}}$. Since the restriction of $\rho-1$ at $G_{K_{\mathfrak{N}_i}}$ is of the form $\eta \cdot \begin{pmatrix}  \phi+\log & b \\ c & -\log \end{pmatrix}$, we get that 
$$\det\begin{pmatrix}  \phi+\log & b \\ c & -\log \end{pmatrix} = 0 \text{ .}$$
This proves that $bc = -\log(\phi+\log)$. Since $b = ka_i^{-1}\log$, we get $c = -k^{-1}a_i\cdot \log - k^{-1}a_i \cdot \phi$, as wanted.
\end{proof}

 Locally at $\mathfrak{N}_i$, the cocycle associated to $\Sigma_p$ is given by
 \begin{align*}
\kappa_{\Sigma}(g) &= g(e_1)-e_1 \\& 
= \chi_p(g)(1+\eta\log(g))\cdot e_1+\eta c(g)\cdot e_2-e_1
\\& = (\phi(g)+\log(g))\cdot \eta e_1 + c(g)\cdot \eta e_2 \text{ .}
\end{align*}
Similarly, the cocycle associated to $C_p$ is given by
 \begin{align*}
\kappa_{C}(g) &= g(e_2)-e_2 \\&
 = b(g)\cdot \eta e_1 - \log(g)\cdot \eta e_2\text{ .}
\end{align*}
So locally at $\mathfrak{N}_i$, the Kummer image is
\[\lambda_{i}\kappa_{C}+\mu_{i}\kappa_{\Sigma}=(\mu_i\phi+\mu_i\log+\lambda_i b)\cdot \eta e_1+(\mu_i c-\lambda_i\log)\cdot \eta e_2\text{ ,}\]
where $\mu_i,\lambda_i\in \Z/p\Z$.
Recall that $\eta e_1 \in \Sigma_p$ corresponds to $\zeta_p \in \mu_p$, which in turns corresponds to $a_i$ in $\Z/p\Z$ via (\ref{eq_coeff_zeta_p_local}).

Therefore, the Kummer image in $$H^1(G_{K_{\mathfrak{N}_i}}, \Z/p\Z) \bigoplus H^1(G_{K_{\mathfrak{N}_i}}, \Z/p\Z) $$
is given by the elements of the form
$$
	\big(a_i \cdot(\mu_i\phi(g)+\mu_i\log(g)+\lambda_i b(g)), \mu_i c(g)-\lambda_i\log(g)\big)
$$
 for $\mu_i,\lambda_i\in \Z/p\Z$.
 
 By Lemma \ref{lemma_description_b_c_g_p=1}, the Kummer image in $$H^1(G_{K_{\mathfrak{N}_i}}, \Z/p\Z) \bigoplus H^1(G_{K_{\mathfrak{N}_i}}, \Z/p\Z) $$
is given by the elements of the form
\begin{equation} \label{kummer image rank 1}
	\big(a_i \mu_i\cdot \phi(g)+(k\lambda_i + a_i\mu_i)\cdot \log(g), -\mu_ia_ik^{-1} \cdot \phi(g) - (\lambda_i+k^{-1}\mu_i a_i)\cdot \log(g)\big)
\end{equation}
 for $\mu_i,\lambda_i\in \Z/p\Z$.

\begin{prop}\label{prop_pic_nontriv_selm}
If $\Pic(\mathcal{O}_K[1/N]) \otimes \Z_p \neq 0$ then $\rk_{\mathbf{F}_p} \Sel_I(\tilde{J}^{(p)}/K) > 1$.
\end{prop}
\begin{proof}
Let us prove that the kernel $V$ of the projection map 
$$ \Sel_I(\tilde{J}^{(p)}/K) \rightarrow H^1(\mathcal{O}_K[1/N], \mu_p)$$ is $\Hom(\Pic(\mathcal{O}_K[1/N]), \Z/p\Z)$. Indeed, let $x \in V$, which we can view inside $H^1(G_{K,N}, \Z/p\Z)$. By (\ref{kummer image rank 1}) the restriction of $x$ to $G_{\mathfrak{N}_i}$ is of the form $ -\mu_ia_ik^{-1} \cdot \phi(g) - (\lambda_i+k^{-1}\mu_i a_i)\cdot \log(g)$ where $\lambda_i, \mu_i \in \Z/p\Z$ are such that 
$$a_i \mu_i\cdot \phi+(k\lambda_i + a_i\mu_i)\cdot \log = 0 \text{ .}$$

This latter equation implies $\mu_i=\lambda_i=0$, so $x$ is locally trivial at $\mathfrak{N}_i$ for $i \in \{1, 2\}$. This proves $x \in \Hom(\Pic(\mathcal{O}_K[1/N]), \Z/p\Z)$. Conversely, $\Hom(\Pic(\mathcal{O}_K[1/N]), \Z/p\Z)$ is contained in $V$ since an element of  $\Hom(\Pic(\mathcal{O}_K[1/N]), \Z/p\Z)$ is locally trivial at $\mathfrak{N}_i$ (and hence in the local Kummer image).

To conclude the proof of the proposition, it suffices to check that the element of $\Sel_I(\tilde{J}^{(p)}/K)$ given by the cuspidal subgroup $C_p$ does not belong to $V$ (\cf Lemma \ref{lem_passage_I_p}). This element is given by the cocycle $\kappa_C$ of (\ref{eq_kappa_C}). Locally at $\mathfrak{N}_i$, we have seen that $\kappa_C$ is given by $(b, -\log)$ and hence is non-trivial. Since all element of $V$ are locally trivial at $\mathfrak{N}_i$, this proves that $\kappa_C \not\in V$. This concludes the proof of Proposition \ref{prop_pic_nontriv_selm}.
\end{proof}

By Proposition \ref{prop_pic_nontriv_selm}, in order to prove (\ref{eq_main_thm_selmer}), it suffices to consider two cases:
\begin{enumerate}[label=(\Alph*)]
\item $p \nmid h(K)$.
\item $p \mid h(K)$ and $\Pic(\mathcal{O}_K[1/N]) \otimes \Z_p = 0$.
\end{enumerate}

Before we treat these two cases, it will be useful to introduce some notation. For $i \in \{1, 2\}$, let 
$$\log_i : \mathcal{O}_{K_{\mathfrak{N}_i}}^{\times} \rightarrow \Z/p\Z$$
be defined as the composition 
$$\mathcal{O}_{K_{\mathfrak{N}_i}}^{\times} \xrightarrow{} \mathbf{F}_N^{\times} \xrightarrow{\log} \Z/p\Z\text{ ,}$$
where the first map is the reduction modulo $\mathfrak{N}_i$.

Let us also denote by $s \in \mathbf{N}$ the order of $\mathfrak{N}_i$ in $\Pic(\mathcal{O}_K)$ (it does not depend on $i$). Write $\mathfrak{N}_i^s = (\pi_i)$ for some $\pi_i \in \mathcal{O}_K$. We can and do assume that $g_0(\pi_1) = \pi_2$ where we recall that $g_0 \in G_{\Q}$ restricts to the non-trivial automorphism of $K$ over $\Q$.

\begin{lem}
We have $\log_1(u(K)) = -\log_2(u(K))$ and $\log_1(\pi_2) = \log_2(\pi_1)$.
\end{lem}
\begin{proof}
Note that if $x \in \mathcal{O}_K$ is coprime with $\mathfrak{N}_1$, then $\log_1(x) = \log_2(g_0(x))$. The first equality then follows from $g_0(u(K)) = \pm u(K)^{-1}$ and the second equality from $g_0(\pi_1) = \pi_2$. 
\end{proof}

The following result deals with case (A).

\begin{prop}\label{prop_case_A}
Assume $p \nmid h(K)$. Then $\rk_{\mathbf{F}_p} \Sel_I(\tilde{J}^{(p)}/K) > 1$ if and only if $\log_1(u(K))=0$.
\end{prop}
\begin{proof}
Recall that we identify $\Sel_I(\tilde{J}^{(p)}/K)$ with a subgroup of 
$$ H^1(\mathcal{O}_K[1/N], \mu_p) \bigoplus H^1(G_{K,N}, \Z/p\Z) \text{ .}$$
An element of $\Sel_I(\tilde{J}^{(p)}/K)$ can thus be written as $(f_{\Sigma}, f_C)$ for some $f_{\Sigma} \in H^1(\mathcal{O}_K[1/N], \mu_p) $ and $f_C \in H^1(G_{K,N}, \Z/p\Z)$.

Let $K(N)$ be the maximal exponent $p$ abelian extension of $K$ unramified outside $\mathfrak{N}_1$ and $\mathfrak{N}_2$. Since $p \nmid h(K)$ and $p\neq N$, we easily see using the idelic Artin map that
\begin{equation}\label{eq_iso_CFT_K_N}
\Gal(K(N)/K) \simeq \left((\mathcal{O}_{K_{\mathfrak{N}_1}}^{\times}/U_1 \times \mathcal{O}_{K_{\mathfrak{N}_2}}^{\times}/U_2) / u(K)^{\Z}\right) \otimes \Z/p\Z\text{ ,}
\end{equation}
where $U_i \subset  \mathcal{O}_{K_{\mathfrak{N}_i}}^{\times}$ is the subgroup of principal units and $u(K)^{\Z}$ means the subgroup of $(\mathcal{O}_K/\mathfrak{N}_1)^{\times} \times (\mathcal{O}_K/\mathfrak{N}_2)^{\times}$ generated by the diagonal image of $u(K)$.

Using $\log_1$ and $\log_2$, we get a canonical isomorphism
\begin{equation}\label{eq_iso_CFT_K_N_2}
\Gal(K(N)/K) \otimes \Z/p\Z \simeq \left(\Z/p\Z \times \Z/p\Z \right)/ \Z\cdot (\log_1(u(K)), -\log_1(u(K))) \text{ .}
\end{equation}
Our homomorphism $f_C \in H^1(G_{K,N}, \Z/p\Z) = \Hom(\Gal(K(N)/K), \Z/p\Z)$ can then be considered as a homomorphism on $\Z/p\Z \times \Z/p\Z$ vanishing on $(\log_1(u(K)), -\log_1(u(K)))$, and is thus determined by its values $v_1 := f_C(1,0)$ and $v_2 := f_C(0,1)$.

The following is a simple exercise in global class field theory. Since $p\nmid h(K)$, we have $p\nmid s$ (where we recall that $s$ is the order of $\mathfrak{N}_i$ in $\Pic(\mathcal{O}_K)$).
\begin{lem}\label{lemma_f_C_v_1v_2}
Let $f_C \in H^1(G_{K,N}, \Z/p\Z)$ and let $v_1,v_2 \in \Z/p\Z$ be defined as above. Locally at $\mathfrak{N}_1$ we have
$$f_C = v_1\cdot \log + s^{-1}\log_1(\pi_2)(v_1-v_2)\cdot \phi$$
and locally at $\mathfrak{N}_2$ we have
$$f_C = v_2\cdot \log + s^{-1} \log_1(\pi_2) (v_2-v_1)\cdot \phi \text{ .}$$
\end{lem}
\begin{proof}[Sketch of proof]
Let us consider $f_C$ locally at $\mathfrak{N}_1$ (the description at $\mathfrak{N}_2$ being obtained by symmetry).
Let $L$ be the maximal exponent $p$ abelian extension of $K_{\mathfrak{N}_1}$. Local class field theory gives a canonical isomorphism
\begin{equation}\label{eq_LCFT_L_1}
(K_{\mathfrak{N}_1}^{\times}/U_1 )\otimes \Z/p\Z \simeq \Gal(L/K_{\mathfrak{N}_1}) \text{ .}
\end{equation}
Note that $L$ is the compositum of $L_1$ and $L_2$, where $L_1$ is the degree $p$ subextension of $K_{\mathfrak{N}_1}(\zeta_N)$ and $L_2$ is the unramified $\Z/p\Z$-extension of $K_{\mathfrak{N}_1}$. We thus have 
\begin{equation}\label{eq_LCFT_L_2}
\Gal(L/K_{\mathfrak{N}_1}) \simeq \Gal(L_1/K_{\mathfrak{N}_1})\times \Gal(L_2/K_{\mathfrak{N}_1}) \simeq \Z/p\Z \times \Z/p\Z \text{ .}
\end{equation}
The isomorphism 
$$(K_{\mathfrak{N}_1}^{\times}/U_1 )\otimes \Z/p\Z \simeq  \Z/p\Z \times \Z/p\Z $$
resulting from (\ref{eq_LCFT_L_1}) and (\ref{eq_LCFT_L_2}) is described as follows. It sends $N \in K_{\mathfrak{N}_1}^{\times}$ to $(0,1)$. Its restriction to $\mathcal{O}_{K_{\mathfrak{N}_1}}^{\times}$ is given by $(\log_1,0)$.

Our fixed embedding 
$$\Gal(\overline{K}_{\mathfrak{N}_1}/K_{\mathfrak{N}_1}) \hookrightarrow G_K$$
yields a canonical map
\begin{equation}\label{eq_local_global_CFT} 
\Gal(L/K_{\mathfrak{N}_1})  \rightarrow \Gal(K(N)/K) \text{ .}
\end{equation}

Let us now describe the map
$$\varphi : \Z/p\Z \times \Z/p\Z \rightarrow  \left(\Z/p\Z \times \Z/p\Z \right)/ \Z\cdot (\log_1(u(K)), -\log_1(u(K))) $$

obtained by combining (\ref{eq_iso_CFT_K_N_2}),  (\ref{eq_LCFT_L_2}) and (\ref{eq_local_global_CFT}). We have $\varphi(1,0)=(1,0)$ (this follows from the compatibility between the local and global Artin maps). Let us now determine $\varphi(0,1)$. It amounts to writing the id\`ele element $x$ of $\mathbb{A}_K^{\times}$ whose components are $1$ everywhere, except $N$ at $\mathfrak{N}_1$, as $x=y\cdot z$ where $y \in K^{\times}$ is a $N$-unit (diagonal element) and $z$ has component $1$ everywhere except possibly at $\mathfrak{N}_1$ and $\mathfrak{N}_2$ where the components of $z$ are units.

Write $N^s = \pm \pi_1\pi_2$. The element $x^s$ is of the form $\pi_1 \cdot y$ where the components of $y$ at places not dividing $N$ are units, the component of $y$ at $\mathfrak{N}_1$ is $\pm \pi_2$ and the component at $\mathfrak{N}_2$ is $\pi_1^{-1}$. Therefore, we have
$$\varphi(0,1) = s^{-1}\cdot (\log_1(\pi_2),-\log_2(\pi_1)) = s^{-1}(\log_1(\pi_2),-\log_1(\pi_2)) \text{ .}$$

By construction, we have locally at $\mathfrak{N}_1$:
\begin{align*}
f_C &= f_C(\varphi(1,0))\cdot \log + f_C(\varphi(0,1))\cdot \phi 
\\& = f_C(1,0)\cdot \log + s^{-1}\log_1(\pi_2)\cdot f_C(1,-1)\cdot \phi 
\\& = v_1\cdot \log + s^{-1} \log_1(\pi_2) (v_1-v_2)\cdot \phi \text{ .}
\end{align*}
This concludes the proof of the lemma.
\end{proof}

Recall that by (\ref{kummer image rank 1}), for each $i \in \{1,2\}$, there exists $\lambda_i, \mu_i \in \Z/p\Z$ such that locally at $\mathfrak{N}_i$
$$f_C = -a_i\mu_ik^{-1} \cdot \phi - (\lambda_i+k^{-1}\mu_i a_i)\cdot \log \text{ .}$$
Thus, the local Kummer condition at $\mathfrak{N}_1$ and $\mathfrak{N}_2$ for $f_C$ can be written as:

\begin{equation}\label{local_kummer_f_C}
	\begin{array}{lcl}
		-a_1\mu_1k^{-1} &=& s^{-1} \log_1(\pi_2) (v_1-v_2),\\
		 - (\lambda_1+k^{-1}\mu_1 a_1) &=& v_1 ,\\
		-a_2\mu_2k^{-1} &=& s^{-1} \log_1(\pi_2) (v_2-v_1),\\
		-(\lambda_2 +k^{-1}\mu_2 a_2)  &=&v_2.
	\end{array}
\end{equation}

Since $p\nmid h(K)$, we have 
$$H^1(\mathcal{O}_K[1/N], \mu_p) \simeq \mathcal{O}_K[1/N]^{\times} \otimes \Z/p\Z \text{ .}$$
An element of $\mathcal{O}_K[1/N]^{\times} \otimes \Z/p\Z$ is of the form $\pi_1^{\alpha_1}\pi_2^{\alpha_2}u(K)^{\beta}$ for some $\alpha_1, \alpha_2, \beta \in \Z/p\Z$.
Thus, our cohomology class $f_{\Sigma}$ is of the form 
$$\pi_1^{\alpha_1}\pi_2^{\alpha_2}u(K)^{\beta} =N^{s\alpha_1} \pi_2^{\alpha_2-\alpha_1}u(K)^{\beta}.$$
 Locally at $\mathfrak{N}_1$, we get
$$f_{\Sigma} = -s\alpha_1\cdot \log+ ((\alpha_2-\alpha_1)\log_1(\pi_2)+\beta\log_1(u(K)))\cdot \phi \text{ .}$$
Similarly, locally at $\mathfrak{N}_2$, we get
\begin{align*}
f_{\Sigma} &= -s\alpha_2\cdot \log+ ((\alpha_1-\alpha_2)\log_2(\pi_1)+\beta\log_2(u(K)))\cdot \phi 
\\& = -s\alpha_2\cdot \log+ ((\alpha_1-\alpha_2)\log_1(\pi_2)-\beta\log_1(u(K)))\cdot \phi \text{ .}
\end{align*}
By (\ref{kummer image rank 1}), the Kummer conditions at $\mathfrak{N}_1$ and $\mathfrak{N}_2$ for $f_{\Sigma}$ can therefore be written as

\begin{equation}\label{local_kummer_f_S}
	\begin{array}{lcl}
		a_1\mu_1  &=& (\alpha_2-\alpha_1)\log_1(\pi_2)+\beta\log_1(u(K)),\\
		 k\lambda_1+a_1\mu_1 &=& -s\alpha_1 ,\\
		a_2\mu_2 &=& (\alpha_1-\alpha_2)\log_1(\pi_2)-\beta\log_1(u(K)),\\
		 k\lambda_2+a_2\mu_2  &=&-s\alpha_2.
	\end{array}
\end{equation}

We now solve equations (\ref{local_kummer_f_C}) and (\ref{local_kummer_f_S}). Note that the variables are $v_1$, $v_2$, $\alpha_1$, $\alpha_2$ and $\beta$. These variables determine completely $f_{C}$ and $f_{\Sigma}$. We distinguish three cases.

Case 1: $\log_1(u(K)) = \log_1(\pi_2)= 0$. By (\ref{local_kummer_f_S}), we get $\mu_1=\mu_2=0$. Our equations reduce to $v_1=-\lambda_1$, $v_2=-\lambda_2$, $\alpha_1 = -s^{-1}k\lambda_1$ and $\alpha_2 = -s^{-1}k\lambda_2$. There is no condition on $\beta$, which can be chosen to be arbitrary. Thus, one can choose $(v_1, v_2, \beta)$ arbitrarily in $(\Z/p\Z)^3$. This shows that in this case $\rk_{\mathbf{F}_p} \Sel_I(\tilde{J}^{(p)}/K) = 3 >1$, as wanted.

Case 2: $\log_1(u(K)) = 0$ and $\log_1(\pi_2) \neq 0$.  By the second and fourth equations of (\ref{local_kummer_f_C}) and (\ref{local_kummer_f_S}), we get $\alpha_1 = ks^{-1} v_1$ and $\alpha_2 = ks^{-1} v_2$. By the first and third equations of (\ref{local_kummer_f_C}), we get $\mu_1 = ka_1^{-1}s^{-1}\log_1(\pi_2)(v_2-v_1)$ and $\mu_2 = -ka_2^{-1}s^{-1}\log_1(\pi_2)(v_2-v_1)$. By the second and fourth equations of (\ref{local_kummer_f_C}), we get $\lambda_1 = -v_1+s^{-1}\log_1(\pi_2)(v_1-v_2)$ and $\lambda_2 = -v_2+s^{-1}\log_1(\pi_2)(v_2-v_1)$. Thus, $\alpha_1$, $\alpha_2$, $\lambda_1$, $\lambda_2$, $\mu_1$ and $\mu_2$ are determined by the choice of $(v_1, v_2, \beta) \in (\Z/p\Z)^3$. Conversely, it is easy to check that the expressions we just gave are solutions of (\ref{local_kummer_f_C}) and (\ref{local_kummer_f_S}). We get again $\rk_{\mathbf{F}_p} \Sel_I(\tilde{J}^{(p)}/K) = 3 >1$, as wanted.

Case 3: $\log_1(u(K)) \neq 0$. In this case, note from (\ref{eq_iso_CFT_K_N_2}) that $f_C$ factors through $\Gal(K(\zeta_N^{(p)})/K)$, where $K(\zeta_N^{(p)})$ is the unique subextension of $K$ of degree $p$ in $K(\zeta_N)$. In other words, we must have $v_1=v_2$. By the first and third equations of (\ref{local_kummer_f_C}), we get $\mu_1=\mu_2=0$. By the second and fourth equations of (\ref{local_kummer_f_C}), we get $\lambda_1=\lambda_2 = -v_1 = -v_2$. By the second and fourth equations of (\ref{local_kummer_f_S}), we get $\alpha_1=\alpha_2=-s^{-1}k\lambda_1 = s^{-1}kv_1$. Finally, by the first and third equations of (\ref{local_kummer_f_S}), we get $\beta = 0$. Thus, we only have one degree of freedom, namely one can choose \eg $v_1$. This proves $\rk_{\mathbf{F}_p} \Sel_I(\tilde{J}^{(p)}/K) = 1$, as wanted. 

This concludes the proof of Proposition \ref{prop_case_A}.

\end{proof}

The following result deals with case (B).

\begin{prop}\label{prop_case_B}
Assume $p \mid h(K)$ and $\Pic(\mathcal{O}_K[1/N]) \otimes \Z_p=0$. Then $\rk_{\mathbf{F}_p} \Sel_I(\tilde{J}^{(p)}/K) > 1$.
\end{prop}
\begin{proof}
The assumptions $p \mid h(K)$ and $\Pic(\mathcal{O}_K[1/N]) \otimes \Z_p=0$ mean that the $p$-part of $\Pic(\mathcal{O}_K)$ is generated by the prime ideals $\mathfrak{N}_1$ and $\mathfrak{N}_2$ above $N$. In particular, $p \mid s$ (where we recall that $s$ is the order of $\mathfrak{N}_i$ in $\Pic(\mathcal{O}_K)$). 

Recall that an element of $\Sel_I(\tilde{J}^{(p)}/K)$ is given by a pair $(f_{\Sigma}, f_C)$ in $$H^1(\mathcal{O}_K[1/N], \mu_p) \bigoplus H^1(G_{K, N}, \Z/p\Z)$$ satisfying the local Kummer conditions (\ref{kummer image rank 1}) at $\mathfrak{N}_1$ and $\mathfrak{N}_2$.

Let us choose $f_C =0$. By (\ref{kummer image rank 1}), the Kummer conditions at $\mathfrak{N}_1$ and $\mathfrak{N}_2$ are 
\begin{equation}
	\begin{array}{lcl}
		-a_1\mu_1k^{-1} &=& 0, \\
		 \lambda_1+k^{-1}\mu_1 a_1 &=& 0 ,\\
		-a_2\mu_2k^{-1} &=& 0,\\
		\lambda_2 +k^{-1}\mu_2 a_2  &=&0.
	\end{array}
\end{equation}

This is equivalent to
\begin{equation}\label{local_kummer_f_C_caseB}
\lambda_1 = \lambda_2 = \mu_1 = \mu_2 = 0 \text{ .}
\end{equation}

Let us show that one can find $f_{\Sigma} \in H^1(\mathcal{O}_K[1/N], \mu_p)$ such that $(f_{\Sigma}, f_C)$ belongs to $ \Sel_I(\tilde{J}^{(p)}/K) $. Recall that by (\ref{eq_Kummer_fppf}) we have 
$$\mathcal{O}_K[1/N]^{\times} \otimes \Z/p\Z \simeq H^1(\mathcal{O}_K[1/N], \mu_p) \text{ .}$$
Since $p \mid s$, we see that an element of $\mathcal{O}_K[1/N]^{\times} \otimes \Z/p\Z$ is of the form $\pi_1^{\alpha} N^{\beta} u(K)^{\gamma}$ for some $\alpha, \beta, \gamma \in \Z/p\Z$. Let us thus write $f_{\Sigma} = \pi_1^{\alpha} N^{\beta} u(K)^{\gamma}$. 

The restriction of $f_{\Sigma}$ at $G_{K_{\mathfrak{N}_2}}$ is
$$f_{\Sigma} = -\beta \log + (\alpha \log_2(\pi_1) + \gamma \log_2(u(K))) \phi =  -\beta \log + (\alpha \log_1(\pi_2) - \gamma \log_1(u(K))) \phi  \text{ .}$$

At $\mathfrak{N}_1$, write $f_{\Sigma} = \pi_2^{-\alpha} N^{\beta+s\alpha} u(K)^{\gamma} = \pi_2^{-\alpha} N^{\beta} u(K)^{\gamma}$ (since $p \mid s$).
The restriction of $f$ at $G_{K_{\mathfrak{N}_1}}$ is therefore
$$f_{\Sigma} = -\beta \log + (-\alpha \log_1(\pi_2) + \gamma \log_1(u(K))) \phi \text{ .}$$

By (\ref{kummer image rank 1}) the local Kummer conditions at $\mathfrak{N}_1$ and $\mathfrak{N}_2$ are

\begin{equation}\label{local_kummer_f_S_caseB}
	\begin{array}{lcl}
		a_1\mu_1 &=& -\alpha \log_1(\pi_2) + \gamma \log_1(u(K)), \\
		 k\lambda_1+\mu_1 a_1 &=& -\beta ,\\
		a_2\mu_2 &=&  \alpha \log_1(\pi_2) - \gamma \log_1(u(K)),\\
		k\lambda_2 + \mu_2 a_2  &=& -\beta.
	\end{array}
\end{equation}
By (\ref{local_kummer_f_C_caseB}), equations (\ref{local_kummer_f_S_caseB}) are equivalent to 
\begin{equation}
	\begin{array}{lcl}
		\beta &=& 0, \\
		 \alpha \log_1(\pi_2) - \gamma \log_1(u(K)) &=&0 .
		 
	\end{array}
\end{equation}

There is always a choice of $(\alpha,  \beta, \gamma)$ which satisfies these equations. Thus, we have proved that there exists an element of $ \Sel_I(\tilde{J}^{(p)}/K) $ of the form $(f_{\Sigma}, 0)$. This element is not a multiple of the cuspidal subgroup $\kappa_C \in \Sel_I(\tilde{J}^{(p)}/K) $. Indeed, locally at $\mathfrak{N}_i$, we have seen that $\kappa_c$ is equal to $(b, -\log)$. This is non-zero so cannot be a multiple of $(f_{\Sigma}, 0)$. This proves that $\rk_{\mathbf{F}_p} \Sel_I(\tilde{J}^{(p)}/K) > 1$, as wanted.

\end{proof}

\subsubsection{Case $g_p>1$}
Assume in this paragraph that $g_p>1$. In this case, we have an isomorphism $\mathbb{T}_{\mathfrak{P}}^0/I^2 \simeq \mathbb{F}_p[T]/(T^2)$ ($T$ corresponding to our local generator $\eta \in I$).  The cyclotomic character (modulo $p$) $\chi_p : G_{K_{\mathfrak{N}_i}} \rightarrow (\Z/p\Z)^{\times}$ is the trivial character. This will make the computations a bit simpler than in the case $g_p=1$.

We shall need the following local description of the cocycles $b$ and $c$ involved in (\ref{eq_description_rho}).
\begin{lem}\label{lemma_description_b_c_g_p>1}
There exists $k\in (\Z/p\Z)^{\times}$ such that for any $i \in \{1,2\}$, the restriction to $G_{K_{\mathfrak{N}_i}}$ of $b$ and $c$ satisfy
$$b = ka_i^{-1}\cdot \log$$
and
$$c = -k^{-1}a_i\cdot \log  \text{ .}$$
\end{lem}
\begin{proof}
A similar argument as in the proof of Lemma \ref{lemma_description_b_c_g_p=1} shows that $b$ represents a non-trivial class in $H^1_{\fppf}(\Z[1/N], \mu_p)$, which must correspond to the Kummer class of $N \otimes k \in \Z[1/N]^{\times} \otimes \Z/p\Z$ for some $k \in (\Z/p\Z)^{\times}$. By Lemma \ref{lemma_log_N_comparison}, we have $a_i \cdot b = - k \cdot \log$ locally at $\mathfrak{N}_i$. Up to replacing $k$ with $-k$, we get $b = ka_i^{-1} \cdot \log$, as wanted.

Let us now prove that $c = -k^{-1}a_i\cdot \log$ locally at $\mathfrak{N}_i$. Recall that $\rho : G_{\Q} \rightarrow \GL_2(\mathbb{T}_{\mathfrak{P}}^0/I^2)$ is given by $\rho = \begin{pmatrix} 1+\eta\cdot \log & \eta b \\ \eta c & 1-\eta \cdot \log \end{pmatrix}$. Since $J_0(N)$ has semi-stable reduction at $N$, we know that there is a quotient line in $\tilde{J}^{(p)}[I^2]$ fixed by $G_{K_{\mathfrak{N}_i}}$. Since the restriction of $\rho-1$ at $G_{K_{\mathfrak{N}_i}}$ is of the form $\eta \cdot \begin{pmatrix}  \log & b \\ c & -\log \end{pmatrix}$, we get that 
$$\det\begin{pmatrix}  \log & b \\ c & -\log \end{pmatrix} = 0 \text{ .}$$
This proves that $bc = -\log^2$. Since $b = ka_i^{-1}\log$, we get $c = -k^{-1}a_i\cdot \log$, as wanted.
\end{proof}

 Locally at $\mathfrak{N}_i$, the cocycle associated to $\Sigma_p$ is given by
 \begin{align*}
\kappa_{\Sigma}(g) &= g(e_1)-e_1 \\& 
= (1+\eta\log(g))\cdot e_1+\eta c(g)\cdot e_2-e_1
\\& = \log(g) \cdot \eta e_1 + c(g)\cdot \eta e_2 \text{ .}
\end{align*}
Similarly, the cocycle associated to $C_p$ is given by
 \begin{align*}
\kappa_{C}(g) &= g(e_2)-e_2 \\&
 = b(g)\cdot \eta e_1 - \log(g)\cdot \eta e_2 \text{ .}
\end{align*}
So locally at $\mathfrak{N}_i$, the Kummer image is
\[\lambda_{i}\kappa_{C}+\mu_{i}\kappa_{\Sigma}=(\mu_i\log+\lambda_i b)\cdot \eta e_1+(\mu_i c-\lambda_i\log)\cdot \eta e_2 \text{ ,}\]
where $\mu_i,\lambda_i\in \Z/p\Z$.
Recall that $\eta e_1 \in \Sigma_p$ corresponds to $\zeta_p \in \mu_p$, which in turns corresponds to $a_i$ in $\Z/p\Z$ via (\ref{eq_coeff_zeta_p_local}).

Therefore, the Kummer image in $$H^1(G_{K_{\mathfrak{N}_i}}, \Z/p\Z) \bigoplus H^1(G_{K_{\mathfrak{N}_i}}, \Z/p\Z) $$
is given by the elements of the form
$$
	\big(a_i \cdot(\mu_i\log(g)+\lambda_i b(g)), \mu_i c(g)-\lambda_i\log(g)\big)
$$
 for $\mu_i,\lambda_i\in \Z/p\Z$.
 
 By Lemma \ref{lemma_description_b_c_g_p>1}, the Kummer image in $$H^1(G_{K_{\mathfrak{N}_i}}, \Z/p\Z) \bigoplus H^1(G_{K_{\mathfrak{N}_i}}, \Z/p\Z) $$
is given by the elements of the form
\begin{equation} \label{kummer image rank >1}
	\big((k\lambda_i + a_i\mu_i)\cdot \log(g), - (\lambda_i+k^{-1}\mu_i a_i)\cdot \log(g)\big)
\end{equation}
 for $\mu_i,\lambda_i\in \Z/p\Z$.

The following is the analogue of Proposition \ref{prop_pic_nontriv_selm}.
\begin{prop}\label{prop_pic_nontriv_selm_g_p>1}
If $\Pic(\mathcal{O}_K[1/N]) \otimes \Z_p \neq 0$ then $\rk_{\mathbf{F}_p} \Sel_I(\tilde{J}^{(p)}/K) > 1$.
\end{prop}
\begin{proof}
Let us prove that the kernel $V$ of the projection map 
$$ \Sel_I(\tilde{J}^{(p)}/K) \rightarrow H^1(\mathcal{O}_K[1/N], \mu_p)$$ is $\Hom(\Pic(\mathcal{O}_K[1/N]), \Z/p\Z)$. Indeed, let $x \in V$, which we can view inside $H^1(G_{K,N}, \Z/p\Z)$. By (\ref{kummer image rank >1}) the restriction of $x$ to $G_{\mathfrak{N}_i}$ is of the form $$ - (\lambda_i+k^{-1}\mu_i a_i)\cdot \log(g)$$ where $\lambda_i, \mu_i \in \Z/p\Z$ are such that 
$$(k\lambda_i + a_i\mu_i)\cdot \log(g) = 0 \text{ .}$$

This latter equation implies $k\lambda_i + a_i\mu_i=0$, \ie $\lambda_i+k^{-1}\mu_i a_i=0$. Thus, $x$ is locally trivial at $\mathfrak{N}_i$ for $i \in \{1, 2\}$. This proves $$x \in \Hom(\Pic(\mathcal{O}_K[1/N]), \Z/p\Z) \text{ .}$$ 
Conversely, $\Hom(\Pic(\mathcal{O}_K[1/N]), \Z/p\Z)$ is contained in $V$ since an element of  $\Hom(\Pic(\mathcal{O}_K[1/N]), \Z/p\Z)$ is locally trivial at $\mathfrak{N}_i$ (and hence in the local Kummer image).

To conclude the proof of the proposition, it suffices to check that the element of $\Sel_I(\tilde{J}^{(p)}/K)$ given by the cuspidal subgroup $C_p$ does not belong to $V$ (\cf Lemma \ref{lem_passage_I_p}). This is exactly the same argument as in the proof of Proposition \ref{prop_pic_nontriv_selm}.
\end{proof}

By Proposition \ref{prop_pic_nontriv_selm_g_p>1}, in order to prove (\ref{eq_main_thm_selmer}), it suffices to consider two cases:
\begin{enumerate}[label=(\Alph*)]
\item $p \nmid h(K)$
\item $p \mid h(K)$ and $\Pic(\mathcal{O}_K[1/N]) \otimes \Z_p = 0$.
\end{enumerate}

The following result, which is an analogue of Proposition \ref{prop_case_A}, deals with case (A).

\begin{prop}\label{prop_case_A>1}
Assume $p \nmid h(K)$. Then $\rk_{\mathbf{F}_p} \Sel_I(\tilde{J}^{(p)}/K) > 1$ if and only if $\log_1(u(K))=0$.
\end{prop}
\begin{proof}
We shall use the same notation as in the proof of Proposition \ref{prop_case_A}, in particular $f_{\Sigma}$, $f_C$, $K(N)$ and $s$.
By (\ref{eq_iso_CFT_K_N_2}), our homomorphism $f_C \in H^1(G_{K,N}, \Z/p\Z) = \Hom(\Gal(K(N)/K), \Z/p\Z)$ can be considered as a homomorphism on $\Z/p\Z \times \Z/p\Z$ vanishing on $(\log_1(u(K)), -\log_1(u(K)))$, and is thus determined by its values $v_1 := f_C(1,0)$ and $v_2 := f_C(0,1)$.

Recall that by (\ref{kummer image rank >1}), for each $i \in \{1,2\}$, there exists $\lambda_i, \mu_i \in \Z/p\Z$ such that locally at $\mathfrak{N}_i$
$$f_C = - (\lambda_i+k^{-1}\mu_i a_i)\cdot \log \text{ .}$$
By Lemma \ref{lemma_f_C_v_1v_2}, the local Kummer condition at $\mathfrak{N}_1$ and $\mathfrak{N}_2$ for $f_C$ can be written as:

\begin{equation}\label{local_kummer_f_C>1}
	\begin{array}{lcl}
		0 &=& s^{-1} \log_1(\pi_2) (v_1-v_2),\\
		 - (\lambda_1+k^{-1}\mu_1 a_1) &=& v_1 ,\\
		0 &=& s^{-1} \log_1(\pi_2) (v_2-v_1),\\
		-(\lambda_2 +k^{-1}\mu_2 a_2)  &=&v_2.
	\end{array}
\end{equation}

Recall that, since $p\nmid h(K)$, our cohomology class $f_{\Sigma}$ is of the form 
$$\pi_1^{\alpha_1}\pi_2^{\alpha_2}u(K)^{\beta} =N^{s\alpha_1} \pi_2^{\alpha_2-\alpha_1}u(K)^{\beta}.$$
 Locally at $\mathfrak{N}_1$, we get
$$f_{\Sigma} = -s\alpha_1\cdot \log+ ((\alpha_2-\alpha_1)\log_1(\pi_2)+\beta\log_1(u(K)))\cdot \phi \text{ .}$$
Similarly, locally at $\mathfrak{N}_2$, we get
$$
f_{\Sigma} = -s\alpha_2\cdot \log+ ((\alpha_1-\alpha_2)\log_1(\pi_2)-\beta\log_1(u(K)))\cdot \phi \text{ .}
$$
By (\ref{kummer image rank >1}), the Kummer conditions at $\mathfrak{N}_1$ and $\mathfrak{N}_2$ for $f_{\Sigma}$ can therefore be written as

\begin{equation}\label{local_kummer_f_S>1}
	\begin{array}{lcl}
		0  &=& (\alpha_2-\alpha_1)\log_1(\pi_2)+\beta\log_1(u(K)),\\
		 k\lambda_1+a_1\mu_1 &=& -s\alpha_1 ,\\
		0 &=& (\alpha_1-\alpha_2)\log_1(\pi_2)-\beta\log_1(u(K)),\\
		 k\lambda_2+a_2\mu_2  &=&-s\alpha_2.
	\end{array}
\end{equation}

We now solve equations (\ref{local_kummer_f_C>1}) and (\ref{local_kummer_f_S>1}). Note that the variables are $v_1$, $v_2$, $\alpha_1$, $\alpha_2$ and $\beta$. These variables determine completely $f_{C}$ and $f_{\Sigma}$. We distinguish three cases.

Case 1: $\log_1(u(K))= \log_1(\pi_2)=0$. Our equations (\ref{local_kummer_f_C>1}) and (\ref{local_kummer_f_S>1}) are equivalent to 

$$
	\begin{array}{lcl}
		v_1  &=& k^{-1}s \alpha_1,\\
		v_2 &=& k^{-1}s \alpha_2 ,\\
		k\lambda_1 + a_1\mu_1 &=& -s \alpha_1, \\
		k\lambda_2 + a_2\mu_2 &=& -s \alpha_2 .
	\end{array}
$$
Thus, one can choose $(\alpha_1, \alpha_2, \beta)$ in $(\Z/p\Z)^3$ freely. This shows that in this case $\rk_{\mathbf{F}_p} \Sel_I(\tilde{J}^{(p)}/K) = 3 >1$, as wanted.

Case 2: $\log(u(K)) = 0$ and $\log_1(\pi_2) \neq 0$.  Our equations (\ref{local_kummer_f_C>1}) and (\ref{local_kummer_f_S>1}) are equivalent to
$$
	\begin{array}{lcl}
		v_1  &=& v_2,\\
		v_1  &=& k^{-1}s \alpha_1,\\
		 \alpha_1 &=& \alpha_2 ,\\
		 k\lambda_1 + a_1\mu_1 &=& -s \alpha_1 \\
		k\lambda_1 + a_1\mu_1 &=& k\lambda_2 + a_2\mu_2.
	\end{array}
$$
Thus, one can choose $(\alpha_1, \beta)$ in $(\Z/p\Z)^2$ freely. This shows that in this case $\rk_{\mathbf{F}_p} \Sel_I(\tilde{J}^{(p)}/K) = 2 >1$, as wanted.

Case 3: $\log_1(u(K)) \neq 0$. In this case, note from (\ref{eq_iso_CFT_K_N_2}) that $f$ factors through $\Gal(K(\zeta_N^{(p)})/K)$, where $K(\zeta_N^{(p)})$ is the unique subextension of $K$ of degree $p$ in $K(\zeta_N)$. In other words, we must have $v_1=v_2$. 
Our equations equations (\ref{local_kummer_f_C>1}) and (\ref{local_kummer_f_S>1}) are equivalent to
$$
	\begin{array}{lcl}
		v_1  &=& v_2,\\
		k\lambda_1 + a_1\mu_1 &=& k\lambda_2 + a_2\mu_2,\\ 
		k\lambda_1 + a_1\mu_1 &=& -s \alpha_1,\\ 
		\alpha_1 &=& \alpha_2,\\
		v_1 &=& k^{-1}s \alpha_1,\\
		\beta &=& 0.
	\end{array}
$$
Thus, in this case, we have only one degree of freedom, namely $v_1$. This shows that in this case $\rk_{\mathbf{F}_p} \Sel_I(\tilde{J}^{(p)}/K) = 1$, as wanted.

This concludes the proof of Proposition \ref{prop_case_A}.

\end{proof}

The following result deals with case (B), and thus concludes the proof of Theorem \ref{main_thm_bis}.

\begin{prop}\label{prop_case_B>1}
Assume $p \mid h(K)$ and $\Pic(\mathcal{O}_K[1/N]) \otimes \Z_p=0$. Then $\rk_{\mathbf{F}_p} \Sel_I(\tilde{J}^{(p)}/K) > 1$.
\end{prop}
\begin{proof}
The assumptions $p \mid h(K)$ and $\Pic(\mathcal{O}_K[1/N]) \otimes \Z_p=0$ mean that the $p$ part of $\Pic(\mathcal{O}_K)$ is generated by the prime ideals $\mathfrak{N}_1$ and $\mathfrak{N}_2$ above $N$. In particular, $p \mid s$ (where we recall that $s$ is the order of $\mathfrak{N}_i$ in $\Pic(\mathcal{O}_K)$). 

Recall that an element of $\Sel_I(\tilde{J}^{(p)}/K)$ is given by a pair $(f_{\Sigma}, f_C)$ in $$H^1(\mathcal{O}_K[1/N], \mu_p) \bigoplus H^1(G_{K, N}, \Z/p\Z)$$ satisfying the local Kummer conditions (\ref{kummer image rank >1}) at $\mathfrak{N}_1$ and $\mathfrak{N}_2$.

Let us choose $f_C =0$. By (\ref{kummer image rank >1}), the Kummer conditions at $\mathfrak{N}_1$ and $\mathfrak{N}_2$ are 
\begin{equation}\label{local_kummer_f_C_caseB>1}
	\begin{array}{lcl}
		 \lambda_1+k^{-1}\mu_1 a_1 &=& 0 ,\\
		\lambda_2 +k^{-1}\mu_2 a_2  &=&0.
	\end{array}
\end{equation}

Let us show that one can find $f_{\Sigma} \in H^1(\mathcal{O}_K[1/N], \mu_p)$ such that $(f_{\Sigma}, f_C)$ belongs to $ \Sel_I(\tilde{J}^{(p)}/K) $. As in the proof of Proposition \ref{prop_case_B}, let us thus write $f_{\Sigma} = \pi_1^{\alpha} N^{\beta} u(K)^{\gamma}$. 

The restriction of $f_{\Sigma}$ at $G_{K_{\mathfrak{N}_2}}$ is
$$f_{\Sigma} = -\beta \log + (\alpha \log_2(\pi_1) + \gamma \log_2(u(K))) \phi =  -\beta \log + (\alpha \log_1(\pi_2) - \gamma \log_1(u(K))) \phi  \text{ .}$$
The restriction of $f$ at $G_{K_{\mathfrak{N}_1}}$ is
$$f_{\Sigma} = -\beta \log + (-\alpha \log_1(\pi_2) + \gamma \log_1(u(K))) \phi \text{ .}$$

By (\ref{kummer image rank >1}) the local Kummer conditions at $\mathfrak{N}_1$ and $\mathfrak{N}_2$ are

\begin{equation}\label{local_kummer_f_S_caseB>1}
	\begin{array}{lcl}
		0 &=& -\alpha \log_1(\pi_2) + \gamma \log_1(u(K)), \\
		 k\lambda_1+\mu_1 a_1 &=& -\beta ,\\
		0 &=&  \alpha \log_1(\pi_2) - \gamma \log_1(u(K)),\\
		k\lambda_2 + \mu_2 a_2  &=& -\beta.
	\end{array}
\end{equation}
By (\ref{local_kummer_f_C_caseB>1}), equations (\ref{local_kummer_f_S_caseB>1}) are equivalent to 
\begin{equation}
	\begin{array}{lcl}
		\beta &=& 0, \\
		 \alpha \log_1(\pi_2) - \gamma \log_1(u(K)) &=&0,\\
		  \lambda_1+k^{-1}\mu_1 a_1 &=& 0 ,\\
		\lambda_2 +k^{-1}\mu_2 a_2  &=&0.
	\end{array}
\end{equation}

There is always a choice of $(\alpha,  \beta, \gamma)$ and $(\lambda_1, \lambda_2, \mu_1, \mu_2)$ which satisfies these equations. Thus, we have proved that there exists an element of $ \Sel_I(\tilde{J}^{(p)}/K) $ of the form $(f_{\Sigma}, 0)$. This element is not a multiple of the cuspidal subgroup $\kappa_C \in \Sel_I(\tilde{J}^{(p)}/K) $. Indeed, locally at $\mathfrak{N}_i$, we have seen that $\kappa_c$ is equal to $(b, -\log)$. This is non-zero so cannot be a multiple of $(f_{\Sigma}, 0)$. This proves that $\rk_{\mathbf{F}_p} \Sel_I(\tilde{J}^{(p)}/K) > 1$, as wanted.

\end{proof}

\bibliography{biblio}
\bibliographystyle{plain}
\newpage

\end{document}